\documentclass[11pt,letterpaper,twoside,reqno,nosumlimits]{amsart}

\usepackage[usenames,dvipsnames]{xcolor}
\usepackage{fancyhdr}
\usepackage{amsmath,amsfonts,amsbsy,amsgen,amscd,amssymb,amsthm}
\usepackage{url}

\usepackage{mathtools}

\usepackage[font=small,margin=0.25in,labelfont={bf},labelsep={colon}]{caption}

\usepackage{subcaption}

\usepackage{tikz}
\usepackage{microtype}
\usepackage{enumitem}

\definecolor{dark-gray}{gray}{0.3}
\definecolor{dkgray}{rgb}{.4,.4,.4}
\definecolor{dkblue}{rgb}{0,0,.5}
\definecolor{medblue}{rgb}{0,0,.75}
\definecolor{rust}{rgb}{0.5,0.1,0.1}

\usepackage[colorlinks=true]{hyperref}

\hypersetup{urlcolor=rust}
\hypersetup{citecolor=dkblue}
\hypersetup{linkcolor=dkblue}

\usepackage{setspace}

\usepackage{graphicx}
\usepackage{booktabs,longtable,tabu} 
\usepackage{multirow} 

\usepackage{float}

\usepackage[full]{textcomp}

\usepackage[scaled=.98,sups]{XCharter}
\usepackage[scaled=1.04,varqu,varl]{inconsolata}
\usepackage[type1]{cabin}
\usepackage[charter,vvarbb,scaled=1.07]{newtxmath}
\usepackage[cal=boondoxo]{mathalfa}
\usepackage[T1]{fontenc}

\usepackage{bm} 

\graphicspath{{figures/}}

\newtheorem{theorem}{Theorem}[section]

\theoremstyle{definition}

\numberwithin{equation}{section} 
\numberwithin{figure}{section}
\numberwithin{table}{section}

\floatstyle{plaintop}
\newfloat{recipe}{thp}{lor}
\floatname{recipe}{Recipe}
\numberwithin{recipe}{section}

\providecommand{\mathbold}[1]{\bm{#1}}  
                                
\renewcommand{\phi}{\varphi}

\newcommand{\iunit}{\mathrm{i}}

\providecommand{\mathbbm}{\mathbb} 

\newcommand{\R}{\mathbbm{R}}

\newcommand{\sgn}[1]{\operatorname{sgn}{#1}}
\newcommand{\real}{\operatorname{Re}}
\newcommand{\imag}{\operatorname{Im}}

\newcommand{\diff}[1]{\mathrm{d}{#1}}
\newcommand{\idiff}[1]{\, \diff{#1}}

\newcommand{\mtx}[1]{\mathbold{#1}}

\newcommand{\triplenorm}[1]{{\left\vert\kern-0.25ex\left\vert\kern-0.25ex\left\vert #1
    \right\vert\kern-0.25ex\right\vert\kern-0.25ex\right\vert}}

\allowdisplaybreaks

\usepackage[numbers]{natbib}

\newcommand{\om}{\omega}

\renewcommand{\th}{\theta}

\mathtoolsset{showonlyrefs}

\begin{document}

\title[Multi-scale Finite-time Blowups of the Constantin--Lax--Majda Model]{Multi-scale Self-similar Finite-time Blowups of\\ the Constantin--Lax--Majda Model for the 3D Euler Equations}
\author[D. Huang, X. Qin, X. Wang]{De Huang$^1$, Xiang Qin$^2$, Xiuyuan Wang$^3$}
\thanks{$^1$ School of Mathematical Sciences, Peking University. Email: dhuang@math.pku.edu.cn}
\thanks{$^2$ Computing and Mathematical Sciences, California Institute of Technology. Email: xqin2@caltech.edu}
\thanks{$^3$ School of Mathematical Sciences, Peking University. Email: wangxiuyuan@stu.pku.edu.cn}

\begin{abstract}
We construct a new class of asymptotically self-similar finite-time blowups that have two collapsing spatial scales for the 1D Constantin--Lax--Majda model. The larger spatial scale measures the decreasing distance between the bulk of the solution and the eventual blowup point, while the smaller scale measures the shrinking size of the bulk of the solution. Our construction is based on an asymptotic analysis applied to the explicit solution formula of the model, and a complex analysis is performed to provide an understanding of multi-scale blowups from a pole dynamics perspective. Similar multi-scale blowup phenomena have recently been discovered for many higher dimensional equations. Our study may provide some understanding of the common mechanism behind these multi-scale blowups.
\end{abstract}

\maketitle

\section{Introduction}
We aim to investigate more complicated finite-time blowup phenomena of the 1D Constantin--Lax--Majda (CLM) model:
\begin{equation}\label{eqt:CLM}
\om_t = \om\mtx{H}(\om),
\end{equation}
where $\mtx{H}$ denotes the Hilbert transform on the real line $\R$. This equation was first introduced by Constantin--Lax--Majda \cite{constantin1985simple} to model the vortex stretching effect in the vorticity formulation of the 3D incompressible Euler equations. One important finding about \eqref{eqt:CLM} in \cite{constantin1985simple} is an explicit formula of the solution $\om(x,t)$ from any suitable initial data $\om(x,0) = \om_0(x)$. Based on this explicit solution formula, finite-time blowup was established under a simple condition for the initial data $\om_0$: The set $\mathrm{S} = \{x:\, \om_0(x)=0\ \text{and}\ \mtx{H}(\om_0)(x)>0\}$ is non-empty. Their result demonstrates that the nonlinear non-local vortex stretching can actually lead to finite-time blowup of the solution. 

Understanding of the blowup phenomena of \eqref{eqt:CLM} was furthered by Elgindi--Jeong \cite{elgindi2020effects}, as they discovered exact self-similar finite-time blowup solutions of the form 
\begin{equation}\label{eqt:one-scale_blowup}
\om(x,t) = (T-t)^{c_\om}\cdot \Omega\left(\frac{x}{(T-t)^{c_l}}\right),
\end{equation}
where $\Omega$ is referred to as the self-similar profile, $T>0$ is some finite blowup time, and $c_\om,c_l$ are the blowup powers for the magnitude and the spatial scaling, respectively. In particular, in the construction of Elgindi--Jeong, $\Omega(x) = -2x/(1+x^2)$ (up to rescaling), $c_\om=-1$, and $c_l=1$. Later, Elgindi--Ghoul--Masmoudi \cite{elgindi2021stable} improved on this result by establishing the stability of such self-similar blowup. Roughly speaking, solutions that are sufficiently close to an exact profile will also blowup in an asymptotically self-similar way as $t\rightarrow T-0$:
\begin{equation}\label{eqt:one-scale_blowup_asymptotic}
\om(x,t) = (T-t)^{c_\om} \left(\Omega\left(\frac{x}{(T-t)^{c_l}}\right) + o(1)\right).
\end{equation}
Suppose that $x=0$ is the only element of the set $\mathrm{S}$. Then an asymptotically self-similar blowup \eqref{eqt:one-scale_blowup_asymptotic} will stably happen around $x=0$ if $\om_0$ satisfies an additional condition that $\om_0'(0)<0$. A similar result was also obtained by Chen--Hou--Huang \cite{chen2021finite} independently. 

Self-similar or asymptotically self-similar finite-time blowups have also been extensively studied for more complicated 1D models for the 3D Euler equations, such as the De Gregorio model \cite{de1996partial,chen2021finite,huang2023self}, the one-parameter family of the generalized CLM models \cite{cordoba2005formation,okamoto2008generalization,castro2010infinite,elgindi2020effects,elgindi2021stable,lushnikov2021collapse,chen2020singularity,huang2024self}, and the 1D Hou--Luo model \cite{luo2014potentially,choi2017finite,chen2022asymptotically,huang2023exact}. These many self-similar finite-time blowups all share the same form in \eqref{eqt:one-scale_blowup} or \eqref{eqt:one-scale_blowup_asymptotic} with the same blowup power $c_\om=-1$ that corresponds to the quadratic nonlinearity of the vortex stretching. Though they may have different profiles $\Omega$ and different spatial scaling $c_l$, there is always one single spatial scaling $(T-t)^{c_l}$ in the blowup.

In this paper, however, we are interested in asymptotically self-similar blowups of \eqref{eqt:CLM} with a two-scale feature:
\begin{equation}\label{eqt:two-scale_blowup}
\om(x,t) = (T-t)^{c_\om}\left(\Omega\left(\frac{x - r(t)\cdot(T-t)^{c_s}}{(T-t)^{c_l}}\right)+ o(1)\right),
\end{equation}
where the larger spatial scale $(T-t)^{c_s}$ (with a smaller power $c_s>0$) measures the distance between the origin and the center location of the bulk part of the solution, while the smaller spatial scale $(T-t)^{c_l}$ (with a larger power $c_l>0$) measures the size of the bulk part of the solution. The function $r(t)$ is continuously differentiable in $t$ and satisfies $0<\lim_{t\rightarrow T-0}r(t)=r(T)<+\infty$; it captures the precise location of the bulk part in the larger scale. That is, the peak location of the solution will converge to the origin like $r(t)(T-t)^{c_s}$, in the meanwhile the width of the peak will shrink even faster like $(T-t)^{c_l}$. It can also be understood as a traveling wave solution that collapses into one point and blows up in finite time. In fact, we will show that the self-similar profile $\Omega$ in the two-scale blowup \eqref{eqt:two-scale_blowup} is an exact traveling wave solution to the CLM model \eqref{eqt:CLM}. Note that a solution of \eqref{eqt:CLM} satisfies the scaling and translation property
\begin{equation}\label{eqt:scaling_property}
\om(x,t)\longrightarrow \alpha \om\left(\beta(x-\gamma),\alpha t\right),
\end{equation}
for any $\alpha,\gamma\in \R$ and $\beta>0$. Hence, one can always rescale the solution so that $T=1$ and $r(T)=1$ (and $\Omega$ will change accordingly). In particular, we will use the explicit solution formula provided in \cite{constantin1985simple} to construct two-scale self-similar blowups of the form \eqref{eqt:two-scale_blowup} with 
\begin{equation}\label{eqt:powers}
c_\om=-(2n+1)/2,\quad c_l=n,\quad c_s=1/2,
\end{equation}
for any positive integer $n\geq 1$ (Theorem \ref{thm:two-scale_general}). Note that in our construction $|c_\om|$ and $c_l$ can be arbitrarily large, while $c_s$ is always $1/2$. Again, suppose that $x=0$ is the only element of the set $\mathrm{S}$. We will show that such a two-scale blowup will happen around $x=0$ given additionally that $\om_0^{(k)}(0)=0$, $k=1,2,\dots,2n$, and $\om_0^{(2n+1)}(0)<0$. In other words, in contrast to the one-scale blowup case, the initial solution $\om_0$ needs to be more degenerate at the blowup point for a two-scale blowup to happen.

It is worth noting that the blowup power $c_\om$ is $-(2n+1)/2$ with $n\geq 1$ in our two-scale cases, meaning that the magnitude is blowing up much faster than that in the one-scale case. Counter-intuitively, this faster growth is actually due to a weaker alignment between $\om$ and $\mtx{H}(\om)$ in the nonlinear stretching term $\om\mtx{H}(\om)$. Simply put, blowing up like $(T-t)^{-1}$ corresponds to the quadratic nonlinearity $\om_t\approx \om^2$, while blowing up like $(T-t)^{-(2n+1)/2}$ corresponds to a weaker nonlinearity $\om_t\approx\om^{(2n+3)/(2n+1)}$. 

To better understand the multi-scale blowup phenomena in a uniform framework, we will perform a complex pole dynamics analysis that has been commonly used in the literature (e.g. \cite{constantin1985simple,elgindi2020effects,lushnikov2021collapse,ambrose2023global,silantyev2024exact}) for studying the CLM model \eqref{eqt:CLM}. By considering the complex function $\eta(z,t) = \om(z,t) + \iunit \mtx{H}(\om)(z,t)$ that solves a local ODE $\eta_t=-\iunit \eta^2/2$, we interpret the finite-time blowup of $\eta(z,t)$ on the real line $\R$ as one or several of its poles touch the real line for the first time. We will construct representative examples of $\eta$ whose poles all start in the lower half of the complex space $\mathbb{C}$ but then travel to the real line with different types of trajectories. Depending on the asymptotic trajectories of the moving poles, the restriction of $\eta$ on $\R$ will exhibit self-similar finite-time blowups with different scale features.

Our motivation of studying multi-scale self-similar blowups arose from recent numerical observations of multi-scale blowups of the 3D Euler equations and its models. One important example is the 2D Boussinesq equations in the half-space $(x_1,x_2)\in\R\times\R_+$, 
\begin{equation}\label{eqt:Boussinesq}
\begin{split}
&\om_t + u_1\om_{x_1} + u_2\om_{x_2} = \th_{x_1},\\
&\th_t + u_1\th_{x_1} + u_2\th_{x_2} = 0,\\
&(u_1,u_2) = \nabla^{\perp}(-\Delta)^{-1}\om,
\end{split}
\end{equation}
where the line $x_2=0$ is a solid boundary. Finite-time singularity focusing at the origin of this system has been numerically discovered for initial data that satisfy certain symmetry and sign conditions. Recently, Chen--Hou \cite{chen2022stable} established the existence of asymptotically self-similar finite-time blowups of the one-scale form \eqref{eqt:one-scale_blowup_asymptotic} using a computer-assisted proof. Remarkably, their result also implies the asymptotically self-similar finite-time blowup of the 3D axisymmetric Euler equations with solid boundary. In addition to some symmetry and sign conditions on the initial data $\om_0, \th_0$ (which ensures $\om_0(0,x_2)=\th_0(0,x_2)=\th_0'(0,x_2)=0$), their proof also requires that $\om_0'(0,0)>0$, $\th_0''(0,0)>0$. However, Liu \cite{liu2017spatial} first numerically observed that, if the initial data instead satisfy a more degenerate condition at the origin,
\[\om_0^{(k)}(0,0)=0,\ k=0,1,2,\quad \om_0^{(3)}(0,0)>0,\quad \th_0^{(k)}(0,0)=0,\ k=0,1,2,3,\quad \th_0^{(4)}(0,0)>0,\]
then the solution of \eqref{eqt:Boussinesq} will develop a two-scale self-similar finite-time blowup of the form \eqref{eqt:two-scale_blowup} focusing at the origin.

A similar scenario happens for the 1D Hou--Luo model proposed in \cite{luo2014potentially}, which models the behavior of the 2D Boussinesq system \eqref{eqt:Boussinesq} on the solid boundary. While Chen--Hou--Huang \cite{chen2022asymptotically} proved existence of a one-scale self-similar blowup of the form \eqref{eqt:one-scale_blowup_asymptotic} with initial data $\om_0(x)\sim x, \th_0(x)\sim x^2$ near $x=0$, Liu \cite{liu2017spatial} also discovered a two-scale self-similar blowup of the form \eqref{eqt:two-scale_blowup} with more degenerate initial data $\om_0(x)\sim x^3, \th_0(x)\sim x^4$ near $x=0$. First, it means that one dimensional models can actually capture the multi-scale blowup phenomena for higher dimensional equations. Second, all these results seem to indicate one common rule: \textit{While less degenerate data can lead to one-scale self-similar finite-time blowups focusing at one point, more degenerate data may lead to self-similar finite-time blowups with multi-scale features}. Our study in this paper thus tries to use the simple model \eqref{eqt:CLM} to provide some understanding of the mechanism behind the multi-scale blowup phenomena.

Another example that has inspired our study comes from recent numerical computations of Hou--Huang \cite{huang2022potential,huang2023potential} and Hou \cite{hou2023potential} on the 3D axisymmetric Euler equations in $\R^3$. They numerically found that, with carefully designed smooth initial data satisfying some symmetry and sign conditions, the 3D axisymmetric Euler equations can develop a two-scale traveling wave solution that tends to blow up at the origin on the symmetry axis. More precisely, the magnitude of the angular vorticity $\om^\th$ tends to blow up in a finite time $T$, and the bulk park of the profile of $\om^\th$ develops a locally self-similar ring around the symmetry axis (in view of axisymmetry) that tends to shrink to one point. According to their numerical fitting, the radius of the ring shrinks like $(T-t)^{1/2}$, while the thickness of the ring shrinks like $(T-t)^1$. Hence, they conjectured that the angular vorticity $\om^\th$ will blow up in a two-scale self-similar form just like \eqref{eqt:two-scale_blowup} with $c_{\om^\th} = -3/2$, $c_l=1$, and $c_s=1/2$ (according to numerical fitting). It is curious that these scaling powers coincide with those in \eqref{eqt:powers} for $n=1$.

We also want to mention a two-scale phenomenon that appears in the head-on collision of two anti-parallel vortex rings in 3D axisymmetric flows without swirl. In the studies of Childress \cite{childress2008growth} and Childress--Gilbert--Valiant \cite{childress2016eroding}, numerical observations and formal calculations suggest that the vorticity concentrated in the vortex rings grows as $t^{4/3}$ (an infinite-time blowup), where the blowup rate $4/3$ was recently proved to be optimal in such scenario \cite{lim2024optimal}. In the meantime, the radius of each vortex ring grows like $t^{4/3}$, while the cross-sectional area of each ring shrinks like $t^{-2}$, demonstrating a spatial two-scale feature in this blowup process. Remarkably, it is observed that, as time goes to infinity, the cross-sectional region converges to the shape of the Sadovskii vortex-patch dipole \cite{sadovskii1971vortex}, which is a steady traveling wave solution to the 2D incompressible Euler equation. It is worth mentioning that the theoretical existence of the Sadovskii vortex-patch dipole was only recently proved by two groups, Huang--Tong \cite{huang2024steady} and Choi--Jeong--Sim \cite{choi2024existence}, using two completely different approaches (one based on a fixed-point method and the other based on a variational method).

As we know, it remains open whether the 3D incompressible Euler equations on $\R^3$ can develop finite-time blowup from smooth initial data, not to mention self-similar ones. So far, the only self-similar finite-time blowup of the 3D incompressible Euler equations on $\R^3$ was constructed by Elgindi \cite{elgindi2021finite} in the axisymmetric setting from $C^{\alpha}$ initial vorticity ($C^{1,\alpha}$ initial velocity) for sufficiently small $\alpha$ (with stability of the blowup established in \cite{elgindi2021stability}). Elgindi's construction corresponds to a one-scale self-similar blowup of the form \eqref{eqt:one-scale_blowup} with $c_\om=-1$ and a profile $\Omega$ that only has $C^\alpha$ regularity at the origin. Unfortunately, no evidence (numerical or theoretical) of one-scale self-similar finite-time blowups of the 3D Euler equations on $\R^3$ with smooth initial data has ever been found. 

Therefore, the numerical observations of Hou--Huang suggest that it might be the correct direction to attack the ultimate problem of the 3D Euler equations by looking for potential self-similar finite-time blowups with multi-scale features. Though the model we consider in this paper is way too simple, we hope that our discussions can still shed some light on the study in this direction.

Finally, we remark that multi-scale finite-time blowups have also been found in other types of PDE systems. For instance, Collot--Ghoul--Masmoudi--Nguyen \cite{collot2023collapsing} recently discovered collapsing-ring finite-time blowup solutions for the $d$-dimensional ($d\geq 3$) Keller--Segel system that have a very similar two-scale structure as in \eqref{eqt:two-scale_blowup}. In their construction, the mass concentrates near a thin sphere shell whose radius shrinks like $(T-t)^{1/d}$ while whose thickness shrinks like $(T-t)^{1-1/d}$ ($T$ is the finite blowup time).

\vspace{-4mm}

\section{Multi-scale blowups of the CLM models}\label{sec:main_resuls}
Our goal is to construct multi-scale self-similar finite-time blowups of the form \eqref{eqt:two-scale_blowup} for the CLM model \eqref{eqt:CLM}. The fundamental tool we will use is the explicit solution formula first found in \cite{constantin1985simple}: for any suitable initial data $\om_0$ (typically $\om_0\in H^1(\R)$), the solution 
\begin{equation}\label{eqt:explicit_solution}
\begin{split}
\om(x,t) &= \frac{4\om_0(x)}{(2-t\mtx{H}(\om_0)(x))^2+t^2\om_0(x)^2},\\
\mtx{H}(\om)(x,t) &= \frac{2\mtx{H}(\om_0)(x)(2-t\mtx{H}(\om_0)(x))-2t\om_0(x)^2}{(2-t\mtx{H}(\om_0)(x))^2+t^2\om_0(x)^2},
\end{split}
\end{equation}
exactly solves \eqref{eqt:CLM}. As stated in \cite[Corollary 1]{constantin1985simple}, it follows immediately from this formula that the solution $\om(x,t)$ will blow up at a finite-time $T$ if and only if the set 
\[\mathrm{S} = \{x\in\R: \om_0(x)=0,\ \mtx{H}(\om_0)(x)>0\}\]
is not empty, and the blowup time $T$ is given by 
\[T = \frac{2}{\sup\{\mtx{H}(\om_0)(x):x\in\mathrm{S}\}}.\] 
A point $x\in\mathrm{S}$ is a blowup point if $\mtx{H}(\om_0)(x) = \sup\{\mtx{H}(\om_0)(x):x\in\mathrm{S}\}$. 

In what follows, we simply consider the case where $x=0$ is the only element of $\mathrm{S}$, i.e. $\mathrm{S} = \{0\}$, and so it must be the blowup point. We can ensure this by imposing odd symmetry and a sign condition on $\om_0$: $\om_0(-x)=-\om_0(x)$, $\om_0(x)<0$ for $x>0$. As a consequence, $\mtx{H}(\om_0)(x)$ is an even function of $x$, and $\mtx{H}(\om_0)(0)>0$. Note that these properties are preserved by the evolution of \eqref{eqt:CLM}. We will argue that, the solution $\om(x,t)$ will develop different types of self-similar blowups focusing at $x=0$, with different orders of degeneracy of $\om_0$ at $x=0$.

\subsection{One-scale self-similar blowup} Let us first review the one-scale self-similar finite-time blowups of the CLM model that have been established in the past. The following result is due to Elgindi--Jeong \cite{elgindi2020effects}: For any initial data $\om_0(x)$ that takes the form
\[\om_0(x) = -\frac{2a^2cx}{a^2+c^2x^2},\quad \mtx{H}(\om_0)(x) = \frac{2a^3}{a^2+c^2x^2},\]
with some $a,c>0$, the solution $\om(x,t)$ is exactly self-similar in time and blows up at $T=1/a$:
\[\om(x,t) = \frac{1}{1-at}\cdot \om_0\left(\frac{x}{1-at}\right).\] 
To prove this result, one only needs to substitute $\om_0$ and $\mtx{H}(\om_0)$ into \eqref{eqt:explicit_solution} and verify that the solution coincides with $(1-at)^{-1}\om_0(x/(1-at))$. Note that $a$ and $c$ are arbitrary positive constants respecting the scaling property \eqref{eqt:scaling_property} of the solution. An example of this type of one-scale self-similar blowup with initial data $\om_0(x)=-2x/(1+x^2)$ is demonstrated in Figure \ref{fig:one-scale}.

We can extend this one-scale blowup result to more general initial data at the cost of losing exact self-similarity. The following theorem was also established in \cite{elgindi2020effects} using an asymptotic argument. Here, we provide a similar proof for the sake of completeness.

\begin{theorem}[One-scale Self-similar Blowup]\label{thm:one-scale}
Assume that the initial value $\om_0\in H^1(\R)$ satisfies 
\begin{enumerate}
\item $\om_0$ is an odd function of $x$;
\item $x=0$ is the unique zero of $\om_0$, and $\sgn(\om_0(x))=-\sgn(x)$;
\item $\om_0\in C^3(-\delta,\delta)$ for some $\delta>0$ and $\om_0'(0) < 0$.
\end{enumerate}
Then, $\mtx{H}(\om_0)(0)>0$. Further assume that 
\[\mtx{H}(\om_0)(0) = 2a,\quad \om_0'(0)=-2c,\]
for some $a,c>0$. Then, as $t\rightarrow T-0$, for $x=O((T-t)^{1/2})$, 
\[
\begin{split}
\om(x,t) &= \frac{1}{T-t}\left(\Omega_1\left(\frac{x}{T-t}\right)+ O(T-t)\right),\\
\mtx{H}(\om)(x,t) &= \frac{1}{T-t}\left(\mtx{H}(\Omega_1)\left(\frac{x}{T-t}\right)+ O(T-t)\right),
\end{split}
\]
where 
\[T = \frac{1}{a},\quad \Omega_1(z) = - \frac{2a^2cz}{a^4 + c^2z^2},\quad \mtx{H}(\Omega_1)(z) = \frac{2a^4}{a^4+c^2z^2}.\]
\end{theorem}

\begin{proof} Our proof is based on the explicit solution formula \eqref{eqt:explicit_solution}. By the assumptions on $\om_0$, we have
\[\mtx{H}(\om_0)(0) = -\frac{1}{\pi}\int_{-\infty}^{\infty}\frac{\om_0(y)}{y}\idiff y >0.\]
Hence, there exist some $a,c>0$ such that
\[\mtx{H}(\om_0)(0) = 2a,\quad \om_0'(0)=-2c.\]
Then, for $x$ sufficiently close to $0$, we have
\[\om_0(x) = -2cx + O(x^3),\quad \mtx{H}(\om_0)(x) = 2a + O(x^2).\]
Substituting these expansions into \eqref{eqt:explicit_solution} yields, for $x$ sufficiently close to $0$,
\begin{align*}
\om(x,t) &= \frac{-8cx + O(x^3)}{4(1-at+ O(tx^2))^2 + 4t^2c^2x^2 + O(t^2x^4)},\\
\mtx{H}(\om(x,t)) &= \frac{8(a+O(x^2))(1-at + O(tx^2)) + O(tx^2)}{4(1-at+ O(tx^2))^2 + 4t^2c^2x^2 + O(t^2x^4)}.
\end{align*}
Let $z = x/(T-t)$, $x = (T-t)z$, where $T=1/a$. Note that $x=O((T-t)^{1/2})$ means $(T-t)^{1/2}z=O(1)$. Then, for $t\in[0,T)$,
\begin{align*}
\om(x,t) &= \frac{-8c(T-t)z + O((T-t)^3z^3)}{4a^2(T-t)^2 + 4c^2T^2(T-t)^2z^2 + O((T-t)^3z^2) + O((T-t)^4z^4)}\\
&= - \frac{1}{T-t}\cdot \frac{8cz + O((T-t)z)}{4a^2 + 4c^2T^2z^2 + O((T-t)z^2)}\\
&= \frac{1}{T-t}\cdot \left( - \frac{2cz}{a^2 + c^2T^2z^2} + O\left((T-t)\cdot \frac{z+z^2}{1+z^2}\right)\right)\\
&= \frac{1}{T-t}\cdot \left( - \frac{2a^2cz}{a^4 + c^2z^2} + O(T-t)\right),
\end{align*}
and 
\begin{align*}
\mtx{H}(\om)(x,t) &= \frac{8a^2(T-t) + O((T-t)^2z^2)}{4a^2(T-t)^2 + 4c^2T^2(T-t)^2z^2 + O((T-t)^3z^2) + O((T-t)^4z^4)}\\
&= \frac{1}{T-t}\cdot \frac{8a^2 + O((T-t)z^2)}{4a^2 + 4c^2T^2z^2 + O((T-t)z^2)}\\
&= \frac{1}{T-t}\cdot \left( \frac{2a^2}{a^2 + c^2T^2z^2} + O\left((T-t)\cdot \frac{z^2}{1+z^2}\right)\right)\\
&= \frac{1}{T-t}\cdot \left( \frac{2a^4}{a^4+c^2z^2} + O(T-t)\right).
\end{align*}
This proves the theorem.
\end{proof}

Theorem \ref{thm:one-scale} asserts that, as long as $x=0$ is the only zero of $\om_0$ and $\om_0'(0)<0$, $\mtx{H}(\om_0)(0)>0$, then the solution will still blow up at $x=0$ in an asymptotically self-similar way with $c_\om=-1$ and $c_l=1$, the same scaling as in the exact self-similar case.

\subsection{Two-scale self-similar blowup: Basic case} We now turn to our new result on the two-scale asymptotically self-similar blowups of the CLM model. To induce a two-scale blowup, the only modification we need is increasing the order of degeneracy of $\om_0$ at $x=0$. In this subsection, we first consider a basic case in which $\om_0(x)\sim-x^3$ for $x$ close to $0$, as the proof for this case is relatively easier and already delivers the main idea. The general case with a more degenerate $\om_0(x)$ at $x=0$ will be handled in the next subsection with some extra techniques. 

\begin{theorem}[Two-scale Blowup: Basic Case]\label{thm:two-scale}
Assume that the initial value $\om_0\in H^1(\R)$ satisfies 
\begin{enumerate}
\item $\om_0$ is an odd function of $x$;
\item $x=0$ is the unique zero of $\om_0$, and $\sgn(\om_0(x))=-\sgn(x)$;
\item $\om_0\in C^5(-\delta,\delta)$ for some $\delta>0$, $\om_0'(0) = 0$, and $\om_0^{(3)}(0)<0$.
\end{enumerate}
Then, $\mtx{H}(\om_0)(0)>0$ and $\mtx{H}(\om_0)''(0)>0$. Further assume that 
\[\mtx{H}(\om_0)(0) = 2a,\quad \mtx{H}(\om_0)''(0) = 4b,\quad \om_0^{(3)}(0)=-24 c,\]
for some $a,b,c>0$. Then, as $t\rightarrow T-0$, for $x = r(T-t)^{1/2}+o((T-t)^{1/2})$, 
\[
\begin{split}
\om(x,t) &= \frac{1}{(T-t)^{3/2}}\left(\Omega_2\left(\frac{x - r(T-t)^{1/2}}{T-t}\right)+ o(1)\right),\\
\mtx{H}(\om)(x,t) &= \frac{1}{(T-t)^{3/2}}\left(\mtx{H}(\Omega_2)\left(\frac{x - r(T-t)^{1/2}}{T-t}\right)+ o(1)\right),
\end{split}
\]
where 
\[T = \frac{1}{a},\quad r=\frac{a}{b^{1/2}}, \quad \Omega_2(z) = -\frac{a^3b^{3/2}c}{a^4c^2+b^4z^2},\quad \mtx{H}(\Omega_2)(z) = -\frac{ab^{7/2}z}{a^4c^2+b^4z^2}.\]
In particular, for $x = r(T-t)^{1/2}+O(T-t)$,
\begin{equation}\label{eqt:two-scale_asymptotic_formula}
\begin{split}
\om(x,t) &= \frac{1}{(T-t)^{3/2}}\left(\Omega_2\left(\frac{x - r(T-t)^{1/2}}{T-t}\right)+ O\big((T-t)^{1/2}\big)\right),\\
\mtx{H}(\om)(x,t) &= \frac{1}{(T-t)^{3/2}}\left(\mtx{H}(\Omega_2)\left(\frac{x - r(T-t)^{1/2}}{T-t}\right)+ O\big((T-t)^{1/2}\big)\right).
\end{split}
\end{equation}
\end{theorem}

\begin{proof} Our proof is again based on the explicit formula \eqref{eqt:explicit_solution}. By the assumptions on $\om_0$, we have
\[\mtx{H}(\om_0)(0) = -\frac{1}{\pi}\int_{-\infty}^{\infty}\frac{\om_0(y)}{y}\idiff y >0,\quad \mtx{H}(\om_0)''(0) = -\frac{2}{\pi}\int_{-\infty}^{\infty}\frac{\om_0(y)}{y^3}\idiff y >0.\]
Hence, there exist some $a,b,c>0$ such that
\[\mtx{H}(\om_0)(0) = 2a,\quad \mtx{H}(\om_0)''(0) = 4b,\quad \om_0^{(3)}(0)=-24c.\]
Then, for $x$ sufficiently close to $0$, we have
\[\om_0(x) = -4cx^3 + O(x^5),\quad \mtx{H}(\om_0)(x) = 2a + 2bx^2 +O(x^4).\]
Substituting these expansions into \eqref{eqt:explicit_solution} yields, for $x$ sufficiently close to $0$,
\begin{align*}
\om(x,t) &= \frac{-16cx^3 + O(x^5)}{4(1-at-btx^2 + O(tx^4))^2+16t^2c^2x^6 + O(t^2x^8)},\\
\mtx{H}(\om)(x,t) &= \frac{8(a+bx^2+O(x^4))(1-at-btx^2+O(x^4)) + O(tx^6)}{4(1-at-btx^2 + O(tx^4))^2+16t^2c^2x^6 + O(t^2x^8)}.
\end{align*}
Let 
\[z = \frac{x - r(T-t)^{1/2}}{T-t},\quad x = r(T-t)^{1/2} + (T-t)z,\]
where $T=1/a$, $r=a/b^{1/2}$. Note that $x=r(T-t)^{1/2} + o(1)$ implies $(T-t)^{1/2}z=o(1)$. We first expand a crucial term that bears a critical cancellation: 
\begin{align*}
1-at-btx^2 + O(tx^4) &= a(T-t) - bT\big(r^2(T-t) + 2r(T-t)^{3/2}z + (T-t)^2z^2\big)\\
&\quad + b(T-t)\big(r^2(T-t) + 2r(T-t)^{3/2}z + (T-t)^2z^2\big) + O((T-t)^2)\\
&= -2bTr(T-t)^{3/2}z + O\big((T-t)^2(1+z^2)\big).
\end{align*}
We can see that it is important for $b$ to be positive, so that the leading term $a(T-t) - bTr^2(T-t)$ can cancel out when $x\sim r(T-t)^{1/2}$. Hence, for $t\in[0,T)$ and $(T-t)^{1/2}z=o(1)$,
\begin{align*}
\om(x,t) &= \frac{-16cr^3(T-t)^{3/2} + O((T-t)^2z)}{16b^2T^2r^2(T-t)^3z^2  + 16c^2T^2r^6(T-t)^3 +O((T-t)^{7/2}z(1+z^2))}\\
&= - \frac{1}{(T-t)^{3/2}}\cdot \frac{cr^3 + O((T-t)^{1/2}z)}{c^2T^2r^6+b^2T^2r^2z^2+O((T-t)^{1/2}z(1+z^2))}\\
&= \frac{1}{(T-t)^{3/2}}\cdot \left(-\frac{cr^3}{c^2T^2r^6+b^2T^2r^2z^2} + O\big((T-t)^{1/2}z\big)\right)\\
&= \frac{1}{(T-t)^{3/2}}\cdot \left(-\frac{a^3b^{3/2}c}{a^4c^2+b^4z^2} + o(1)\right),
\end{align*}
and
\begin{align*}
\mtx{H}(\om)(x,t) &= \frac{-16abTr(T-t)^{3/2}z + O((T-t)^2(1+z^2))}{16b^2T^2r^2(T-t)^3z^2 + 16c^2T^2r^6(T-t)^3 +O((T-t)^{7/2}z(1+z^2))}\\
&= - \frac{1}{(T-t)^{3/2}}\cdot \frac{abTrz + O((T-t)^{1/2}(1+z^2))}{c^2T^2r^6+b^2T^2r^2z^2+O((T-t)^{1/2}z(1+z^2))}\\
&= \frac{1}{(T-t)^{3/2}}\cdot \left(-\frac{abTrz}{c^2T^2r^6+b^2T^2r^2z^2} + O\big((T-t)^{1/2}(1+z)\big)\right)\\
&= \frac{1}{(T-t)^{3/2}}\cdot \left(-\frac{ab^{7/2}z}{a^4c^2+b^4z^2} + o(1)\right).
\end{align*}
In particular, for $x=r(T-t)^{1/2} + O(T-t)$, i.e. $z=O(1)$,
\begin{align*}
\om(x,t) &= \frac{1}{(T-t)^{3/2}}\cdot \left(-\frac{a^3b^{3/2}c}{a^4c^2+b^4z^2} + O\big((T-t)^{1/2}\big)\right),\\
\mtx{H}(\om)(x,t) &= \frac{1}{(T-t)^{3/2}}\cdot \left(-\frac{ab^{7/2}z}{a^4c^2+b^4z^2} + O\big((T-t)^{1/2}\big)\right).
\end{align*}
The theorem is thus proved.
\end{proof}

What Theorem \ref{thm:two-scale} describes is a focusing finite-time blowup solution that has two spatial scales. The larger scale $(T-t)^{1/2}$ captures the traveling maximum location of the solution, while the smaller scale $T-t$ measures the characteristic width of the bulk part of the solution. It can be viewed as a traveling wave solution that collapses into the origin and blows up in finite time. An example of this type of two-scale asymptotically self-similar blowup with initial data $\om_0(x)=-2x^3/(4+x^4)$ is demonstrated in Figure \ref{fig:two-scale}.

\subsection{Two-scale self-similar blowup: General case} Now, we consider more general initial data with an arbitrary order of degeneracy at the blowup point. Let us briefly explain the main idea. Suppose that $\om_0(x)\sim -x^{2n+1}$ near $x=0$ for some integer $n\geq 1$. We will use the implicit function theorem to find a moving point $X(t)$ near the origin such that $2-t\mtx{H}(\om_0)(X(t))=0$, and we will show that $|X(t)| \sim (T-t)^{1/2}$ as $t\rightarrow T-0$. Then, by the explicit solution formula \eqref{eqt:explicit_solution}, $|\om(X(t),t)|$ will blow up like $|\om_0(X(t))|^{-1}\sim |X(t)|^{-(2n+1)} \sim (T-t)^{-(2n+1)/2}$. Employing this idea, we can then generalize Theorem \ref{thm:two-scale} to the following.

\begin{theorem}[Two-scale Blowup: General Case]\label{thm:two-scale_general}
Assume that the initial value $\om_0\in H^1(\R)$ satisfies 
\begin{enumerate}
\item $\om_0$ is an odd function of $x$;
\item $x=0$ is the unique zero of $\om_0$, and $\sgn(\om_0(x))=-\sgn(x)$;
\item there is some positive integer $n\geq1$ such that $\om_0\in C^{2n+3}(-\delta,\delta)$ for some $\delta>0$, and 
\[\om_0^{(k)}(0) = 0,\ k=0,1,\dots,2n,\quad \om_0^{(2n+1)}(0)<0.\]
\end{enumerate}
Then, $\mtx{H}(\om_0)(0)>0$ and $\mtx{H}(\om_0)''(0)>0$. Further assume that 
\[\mtx{H}(\om_0)(0) = 2a,\quad \mtx{H}(\om_0)''(0) = 4b,\quad \om_0^{(2n+1)}(0)=-4c(2n+1)!\,,\]
for some $a,b,c>0$. Let 
\[T = \frac{1}{a}, \quad \Omega_{2,n}(z) = -\frac{a^{2n+1}b^{(2n+1)/2}c}{a^{4n}c^2+b^{2n+2}z^2},\quad \mtx{H}(\Omega_{2,n})(z) = -\frac{ab^{(4n+3)/2}z}{a^{4n}c^2+b^{2n+2}z^2}.\]
Then, there exists some function $r(t)\in C^1([t_0,T])$ for some $t_0\in(0,T)$ satisfying 
\[r(T) = \frac{a}{b^{1/2}},\]
such that, as $t\rightarrow T-0$, for $x = r(t)\cdot (T-t)^{1/2}+O((T-t)^n)$, 
\begin{equation}\label{eqt:two-scale_general_formula}
\begin{split}
\om(x,t) &= \frac{1}{(T-t)^{(2n+1)/2}}\left(\Omega_{2,n}\left(\frac{x - r(t)(T-t)^{1/2}}{(T-t)^n}\right)+ O\big((T-t)^{1/2}\big)\right),\\
\mtx{H}(\om)(x,t) &= \frac{1}{(T-t)^{(2n+1)/2}}\left(\mtx{H}(\Omega_{2,n})\left(\frac{x - r(t)(T-t)^{1/2}}{(T-t)^n}\right)+ O\big((T-t)^{1/2}\big)\right).
\end{split}
\end{equation}
\end{theorem}

\begin{proof}
We will follow the steps in the proof of Theorem \ref{thm:two-scale}, except that we need to use the implicit function theorem to track the center position of the bulk of the solution. 

Again, by the assumptions on $\om_0$, we have
\[\mtx{H}(\om_0)(0) = -\frac{1}{\pi}\int_{-\infty}^{\infty}\frac{\om_0(y)}{y}\idiff y >0,\quad \mtx{H}(\om_0)''(0) = -\frac{2}{\pi}\int_{-\infty}^{\infty}\frac{\om_0(y)}{y^3}\idiff y >0.\]
Hence, there exist some $a,b,c>0$ such that
\[\mtx{H}(\om_0)(0) = 2a,\quad \mtx{H}(\om_0)''(0) = 4b,\quad \om_0^{(2n+1)}(0)=-4c(2n+1)!.\]
For $x\in(-\delta,\delta)$, define 
\begin{equation}\label{eqt:p_q_definition}
p(x) := \frac{\om_0(x)}{-4cx^{2n+1}},\quad q(x) := \frac{\mtx{H}(\om_0)(x)-\mtx{H}(\om_0)(0)}{2bx^2} = \frac{1}{2b}\mtx{H}\left(\frac{\om_0}{x^2}\right)(x).
\end{equation}
Then, by the assumptions on $\om_0$, it is easy to find that $p,q\in C^2(-\delta,\delta)$ are both even functions of $x$, and $p(0) = q(0) = 1$. Hence, for any $x_0,x\in (-\delta,\delta)$, we have
\begin{equation}\label{eqt:general_step_1_1}
p(x) = p(x_0) + O(|x-x_0|),\quad q(x) = q(x_0) + O(|x-x_0|),
\end{equation}
and in particular, 
\begin{equation}\label{eqt:general_step_1_2}
p(x) = 1 + O(x^2),\quad q(x) = 1 + O(x^2).
\end{equation}
Moreover, we can make $\delta$ sufficiently small so that $p(x),q(x)\geq 1/2$ for $x\in(-\delta,\delta)$. We then rewrite $\om_0$ and $\mtx{H}(\om_0)$ as 
\[\om_0(x) = -4cx^{2n+1}p(x),\quad \mtx{H}(\om_0)(x) = 2a + 2bx^2q(x).\]
Substituting these expressions into \eqref{eqt:explicit_solution} yields, for $x\in(-\delta,\delta)$,
\begin{equation}\label{eqt:general_step_2}
\begin{split}
\om(x,t) &= \frac{-16cx^{2n+1}p(x)}{4\big(1- at - btx^2q(x)\big)^2+16t^2c^2x^{4n+2}p(x)^2},\\
\mtx{H}(\om)(x,t) &= \frac{8\big(a+bx^2q(x)\big)\big(1- at-btx^2q(x)\big) -32tc^2x^{4n+2}p(x)^2}{4\big(1- at - btx^2q(x)\big)^2+16t^2c^2x^{4n+2}p(x)^2}.
\end{split}
\end{equation} 

Next, we shall find a proper traveling point of $x$ where the leading order terms in $1- at - btx^2q(x)$ cancel to $0$. Let $X = X(s)$ with $X(0)=0$ be the implicit function determined by the equation 
\[X\cdot q(X)^{1/2} = s.\]
Note that 
\[\big(xq(x)^{1/2}\big)'\Big|_{x=0}= \left(q(x)^{1/2} + \frac{xq'(x)}{2q(x)^{1/2}}\right)\Big|_{x=0}= q(0)^{1/2} = 1.\]
Hence, by the implicit function theorem, there is some $s_0$ such that, for $s\in[-s_0,s_0]$, there is some function $X(s)\in C^1([-s_0,s_0])$ satisfying $X(s)\in(-\delta,\delta)$, and
\[X(s)\cdot q(X(s))^{1/2} = s.\]
Note that, since $q\in C^2(-\delta,\delta)$, $q(0)=1$, $q'(0)=0$, we have $X(s)/s = q(X(s))^{-1/2}\in C^1([-s_0,s_0])$, and
\[\lim_{s\rightarrow 0}\frac{X(s)}{s}= \lim_{s\rightarrow 0}X'(s) = \lim_{s\rightarrow 0}\frac{1}{q(X(s))^{1/2}} = 1,\]
\[\lim_{s\rightarrow 0}\frac{1}{s}\left(\frac{X(s)}{s}\right)' = \lim_{s\rightarrow 0}\frac{1}{s}\cdot \frac{-q'(X(s))X'(s)}{2q(X(s))^{3/2}} = -\lim_{s\rightarrow 0}\frac{X'(s)}{2q(X(s))^{3/2}}\cdot \frac{q'(X(s))}{X(s)} \cdot \frac{X(s)}{s}= -\frac{q''(0)}{2}.\]
The limits above imply that $g(\tau) \in C^1([0,s_0^2])$ where $g(\tau):= X(\tau^{1/2})/\tau^{1/2}$. In particular, we have
\[\frac{X(s)}{s} = 1 + O(s^2).\]
Let $T=1/a$, and let $t_0\in(0,T)$ be determined such that 
\[\frac{a(T-t)}{bt} \leq s_0^2\quad \text{for $t\in[t_0,T]$}. \]
Define a function $r(t)$ on $[t_0,T]$ as
\[r(t) := \frac{1}{(T-t)^{1/2}}\cdot X\left(\left(\frac{a(T-t)}{bt}\right)^{1/2}\right).\]
Owing to the properties of $X(s)$, it is straightforward to check that $r(t)\in C^1([t_0,T])$, and it satisfies 
\begin{equation}\label{eqt:general_step_3}
r(t)^2q\big(r(t)(T-t)^{1/2}\big) = \frac{a}{bt},\quad t\in[t_0,T],
\end{equation}
\[r(T) = \lim_{t\rightarrow T-0}\left(\frac{a}{bt}\right)^{1/2}\cdot \left(\frac{bt}{a(T-t)}\right)^{1/2}\cdot X\left(\left(\frac{a(T-t)}{bt}\right)^{1/2}\right) = \left(\frac{a}{bT}\right)^{1/2} = \frac{a}{b^{1/2}},\]
and 
\begin{equation}\label{eqt:general_step_4}
r(t) = r(T) + O(T-t) = \frac{a}{b^{1/2}} + O(T-t).
\end{equation}

Now, let us go back to the expressions of $\om(x,t)$ and $\mtx{H}(\om)(x,t)$ in \eqref{eqt:general_step_2}. Consider the change of variables
\[z = \frac{x-r(t)\cdot (T-t)^{1/2}}{(T-t)^n},\quad x = r(t)\cdot (T-t)^{1/2}+(T-t)^nz.\]
Note that $x = r(t)\cdot (T-t)^{1/2}+O((T-t)^n)$ implies $z = O(1)$. Using the expansions of $q$ in \eqref{eqt:general_step_1_1} and \eqref{eqt:general_step_1_2} we find that, for $z=O(1)$ and for $t$ sufficiently close to $T$,
\begin{align*}
q\big(r(t)(T-t)^{1/2}+(T-t)^nz\big) &= q(r(t)(T-t)^{1/2}) + O((T-t)^n) \\
&= q(0) + O(T-t) + O((T-t)^n) = 1 + O(T-t),
\end{align*}
and similarly, 
\[p\big(r(t)(T-t)^{1/2}+(T-t)^nz\big) = 1 + O(T-t).\]
We then calculate that, for $t\in[t_0,T]$, 
\begin{align*}
1-at-btx^2q(x) &= a(T-t) - bt\cdot \big(r(t)(T-t)^{1/2}+(T-t)^nz\big)^2\cdot q\big(r(t)(T-t)^{1/2}+(T-t)^nz\big)\\
&= a(T-t) - bt\cdot r(t)^2(T-t)\cdot q(r(t)(T-t)^{1/2}) \\
&\quad - 2bt\cdot r(t)(T-t)^{(2n+1)/2}z\cdot q(r(t)(T-t)^{1/2}) + O((T-t)^{n+1})\\
&= - 2bt\cdot r(t)(T-t)^{(2n+1)/2}z\cdot q(r(t)(T-t)^{1/2}) + O((T-t)^{n+1})\\
&= - 2bT\cdot r(t)(T-t)^{(2n+1)/2}z + O((T-t)^{n+1}).
\end{align*}
We have used identity \eqref{eqt:general_step_3} to cancel out the leading terms proportional to $T-t$, that is,
\[a(T-t) - bt\cdot r(t)^2(T-t)\cdot q(r(t)(T-t)^{1/2}) = 0.\]
Therefore, we can proceed with the expressions in \eqref{eqt:general_step_2} to obtain, for $t\in[t_0,T)$ and $z=O(1)$,
\begin{align*}
\om(x,t) &= \frac{-16c r(t)^{2n+1}(T-t)^{(2n+1)/2} + O((T-t)^{2n}) + O((T-t)^{(2n+3)/2})}{16b^2T^2 r(t)^2(T-t)^{2n+1}z^2  + 16c^2T^2 r(t)^{4n+2}(T-t)^{2n+1} +O((T-t)^{(4n+3)/2})}\\
&= - \frac{1}{(T-t)^{(2n+1)/2}}\cdot \frac{cr(t)^{2n+1} + O((T-t)^{(2n-1)/2}) + O((T-t)) }{c^2T^2r(t)^{4n+2}+b^2T^2r(t)^2z^2+O((T-t)^{1/2})}\\
&= - \frac{1}{(T-t)^{(2n+1)/2}}\cdot \frac{cr(T)^{2n+1} + O((T-t)^{1/2}) }{c^2T^2r(T)^{4n+2}+b^2T^2r(T)^2z^2+O((T-t)^{1/2})}\\
&= \frac{1}{(T-t)^{(2n+1)/2}}\cdot \left(-\frac{cr(T)^{2n+1}}{c^2T^2r(T)^{4n+2}+b^2T^2r(T)^2z^2} + O\big((T-t)^{1/2}\big)\right)\\
&= \frac{1}{(T-t)^{(2n+1)/2}}\cdot \left(-\frac{a^{2n+1}b^{(2n+1)/2}c}{a^{4n}c^2+b^{2n+2}z^2} + O\big((T-t)^{1/2}\big)\right),
\end{align*}
and
\begin{align*}
\mtx{H}(\om)(x,t) &= \frac{-16abTr(t)(T-t)^{(2n+1)/2}z + O((T-t)^{n+1})}{16b^2T^2 r(t)^2(T-t)^{2n+1}z^2  + 16c^2T^2 r(t)^{4n+2}(T-t)^{2n+1} +O((T-t)^{(4n+3)/2})}\\
&= - \frac{1}{(T-t)^{(2n+1)/2}}\cdot \frac{abTr(t)z + O((T-t)^{1/2})}{c^2T^2r(t)^{4n+2}+b^2T^2r(t)^2z^2+O((T-t)^{1/2})}\\
&= - \frac{1}{(T-t)^{(2n+1)/2}}\cdot \frac{abTr(T)z + O((T-t)^{1/2})}{c^2T^2r(T)^{4n+2}+b^2T^2r(T)^2z^2+O((T-t)^{1/2})}\\
&= \frac{1}{(T-t)^{(2n+1)/2}}\cdot \left(-\frac{abTr(T)z}{c^2T^2r(T)^{4n+2}+b^2T^2r(T)^2z^2} + O\big((T-t)^{1/2}\big)\right)\\
&= \frac{1}{(T-t)^{(2n+1)/2}}\cdot \left(-\frac{ab^{(4n+3)/2}z}{a^{4n}c^2+b^{2n+2}z^2} + O\big((T-t)^{1/2}\big)\right).
\end{align*}
We have used \eqref{eqt:general_step_1_2} and \eqref{eqt:general_step_4} to expand $p(r(t)(T-t)^{1/2}+(T-t)^nz)$ and $r(t)$ around $t=T$. The theorem is thus proved.
\end{proof}

One can view Theorem \ref{thm:two-scale} as a special case of Theorem \ref{thm:two-scale_general} with $n=1$, and the profile $\Omega_2$ in \eqref{eqt:two-scale_asymptotic_formula} can be recognized as $\Omega_2=\Omega_{2,1}$. That is to say, one can simply choose $r(t) \equiv r(T) = a/b^{1/2}$ in Theorem \ref{thm:two-scale_general} when $n=1$, and the asymptotic formulas in \eqref{eqt:two-scale_general_formula} are still correct. Nevertheless, the function $r(t)$ obtained by the implicit function theorem captures the bulk location of the solution more accurately than the constant $r(T) = a/b^{1/2}$, especially for $n\geq 2$. In fact, we can see from the proof above that $r(t) - r(T) = O(T-t)$. When $n=1$, we have $(r(t) - r(T))(T-t)^{1/2} = O((T-t)^{3/2}) = O((T-t)^{n+1/2})$, so replacing $r(t)$ by $r(T)$ does not compromise the accuracy of the asymptotic formulas \eqref{eqt:two-scale_general_formula}. However, when $n\geq 2$, the smaller scale $(T-t)^n$ is too small so that replacing $r(t)$ by $r(T)$ will lead to a large error in the formulas.

\subsection{Relation to traveling wave solutions} Note that the profile function $\Omega_{2,n}$ in \eqref{eqt:two-scale_general_formula}, or simply a function like $1/(1+x^2)$, is related to non-blowup, exact traveling wave solutions of the CLM model \eqref{eqt:CLM}. More precisely, one can easily check that, for any speed $r\in \R$ and any constant $c>0$, the traveling wave solution
\begin{equation}\label{eqt:traveling_wave}
\om(x,t) = \frac{-2cr}{1+c^2(x+rt)^2}
\end{equation}
exactly solves \eqref{eqt:CLM}. However, for the solution to develop an exact or asymptotic traveling wave that exists for all time, the initial data $\om_0$ must have no blowup points (i.e. the set $\mathrm{S}$ is empty). From this perspective, we see that the two-scale, asymptotically self-similar blowup solution in Theorem \ref{thm:two-scale_general} is a mixture of two phenomena. On the one hand, due to the existence of a unique blowup point (a zero of $\om_0$ with positive $\mtx{H}(\om_0)$), the solution will blow up at this point in finite time driven by non-linear stretching. On the other hand, in the scope of the smaller scale $(T-t)^n$ (which measures the local characteristic scale), the blowup point $x=0$ is asymptotically farther and farther away from the bulk of the solution concentrating near $r(t)(T-t)^{1/2}$, and thus the solution can also behave like a traveling wave.

Let us say a little more about traveling wave solutions. Suppose that $\om_0(x) = f(x)$ is sufficiently smooth, decays sufficiently fast in the far field, and does not have any zero on $\R$. If we require $\om(x,t) = f(x+rt)$, $r\neq0$, to be a solution of \eqref{eqt:CLM}, then $f$ solves 
\begin{equation}\label{eqt:traveling_wave_profile_equation}
rf' = f\mtx{H}(f).
\end{equation} 
Applying the Hilbert transform, using Tricomi's identity $2\mtx{H}(f\mtx{H}f) = \mtx{H}(f)^2 - f^2$, and denoting $g=\mtx{H}(f)$, we get an ODE system
\[
f' = \frac{1}{r}fg,\quad g' = \frac{1}{2r}(g^2-f^2).
\]
Since $f$ does not change sign on $\R$, $g=\mtx{H}(f)$ must change sign on $\R$ (because $\int_{\R}f(x)\mtx{H}(f)(x)\idiff x=0$). Hence, we may assume that $f(0)=a\neq 0$ and $g(0)=0$ (by translation). We then find that 
\[\left(\frac{f}{f^2+g^2}\right)' = 0\quad \Longrightarrow\quad \frac{f(x)}{f(x)^2+g(x)^2} = \frac{f(0)}{f(0)^2+g(0)^2} = \frac{1}{a},\]
and that
\[\left(\frac{g(x)}{f(x)}\right)' = -\frac{1}{2r}\frac{f(x)^2+g(x)^2}{f(x)} = -\frac{a}{2r}\quad \Longrightarrow\quad \frac{g(x)}{f(x)} = -\frac{a}{2r}x.\]
It then easily follows that
\[f(x) = \frac{a}{1+(a/2r)^2x^2},\quad g(x)=\frac{-(a^2/2r)x}{1+(a/2r)^2x^2}.\]
Moreover, the relation $g=\mtx{H}(f)$ requires that $\sgn(a)=-\sgn(r)$. Hence, by writing $a = -2cr$ for some $c>0$, we get 
\[f(x) = \frac{-2cr}{1+c^2x^2}.\]
This gives all non-blowup traveling wave solutions $\om(x,t) = f(x+rt)$ for \eqref{eqt:CLM}. Note that one can also easily solve the ODE system of $(f,g)$ by considering the complex function $f+\iunit g$. Such a complex technique will be discussed in the next section.

Now, we perform an asymptotic analysis to explain why the asymptotic profile $\Omega$ in the two-scale self-similar blowup ansatz \eqref{eqt:two-scale_blowup} should be the profile of an exact traveling wave solution. Let us instate the assumptions in Theorem \ref{thm:two-scale_general} but pretend that we do not know the result. We first note that, at the origin $x=0$,
\[\frac{\diff\,}{\diff t}\mtx{H}(\om)(0,t) = \frac{1}{2}\left(\mtx{H}(\om)(0,t)^2-\om(0,t)^2\right) = \frac{1}{2}\mtx{H}(\om)(0,t)^2,\]
which implies $\mtx{H}(\om)(0,t) = 2(T-t)^{-1}$ with $T = 2/\mtx{H}(\om)(0,0)$. This is probably the simplest way to see the blowup nature of the CLM model. We have thus found a conserved quantity of the solution:
\begin{equation}\label{eqt:conservation}
(T-t)\mtx{H}(\om)(0,t) = -\frac{T-t}{\pi}\int_{\R}\frac{\om(x)}{x}\idiff x \equiv 2,\quad t\in[0,T].
\end{equation}
Next, suppose that we know the solution will develop a two-scale asymptotically self-similar blowup of the form \eqref{eqt:two-scale_blowup} with some $c_\om<0$, $c_l>c_s>0$, some $r(t)\in C^1([0,T])$ with $r(T)>0$, and some suitable $\Omega\in L^1(\R)$. Let us even ignore the $o(1)$ correction in the formula, that is, 
\begin{equation}\label{eqt:two-scale_ansatz}
\om(x,t) = (T-t)^{c_\om}\Omega\left(\frac{x - r(t)(T-t)^{c_s}}{(T-t)^{c_l}}\right).
\end{equation}
Substituting this ansatz into \eqref{eqt:conservation} gives
\begin{align*}
-2\pi &= \int_{\R}\frac{(T-t)^{c_\om+1}}{x}\Omega\left(\frac{x - r(t)(T-t)^{c_s}}{(T-t)^{c_l}}\right)\idiff x\\
&=  \int_{\R}\frac{(T-t)^{c_\om+c_l+1}}{r(t)(T-t)^{c_s} + (T-t)^{c_l}y}\Omega(y)\idiff y\\
&= \frac{(T-t)^{c_\om+c_l-c_s+1}}{r(t)}\int_{\R}\left(\Omega(y) - \frac{(T-t)^{c_l-c_s}y}{r(t) + (T-t)^{c_l-c_s}y}\Omega(y)\right)\idiff y\\
&\sim \frac{(T-t)^{c_\om+c_l-c_s+1}}{r(t)}\int_{\R}\Omega(y)\idiff y.
\end{align*}
For the last integral above to be conserved as $t\rightarrow T-0$, we thus need $c_\om$, $c_l$, $c_s$ to satisfy the relation
\begin{equation}\label{eqt:power_relation}
c_\om + c_l - c_s + 1=0.
\end{equation}
To proceed, we substitute the ansatz \eqref{eqt:two-scale_ansatz} into the original equation \eqref{eqt:CLM} and get
\begin{align*}
&c_\om(T-t)^{c_\om-1}\Omega(z) + c_l(T-t)^{c_\om-1}z\Omega'(z) + r(t)c_s(T-t)^{c_\om+c_s-c_l-1}\Omega'(z) - r'(t)(T-t)^{c_\om+c_s-c_l}\Omega'(z) \\
&= (T-t)^{2c_\om}\Omega(z)\mtx{H}(\Omega)(z),
\end{align*}
where $z = (x-r(t)(T-t)^{c_s})/(T-t)^{c_l}$. Using the relation \eqref{eqt:power_relation}, we can rewrite the equation above as
\[(T-t)^{c_l-c_s}\left(c_\om\Omega + c_lz\Omega'\right) - (T-t)r'(t)\Omega'(z) + r(t)c_s\Omega' - \Omega\mtx{H}(\Omega) = 0.\]
Since it has been assumed that $c_l>c_s>0$, for this equation to hold as $t\rightarrow T-0$, the profile $\Omega$ must satisfy the leading order equation
\[r(T)c_s\Omega' = \Omega\mtx{H}(\Omega),\]
which is exactly the profile equation \eqref{eqt:traveling_wave_profile_equation} for traveling wave solutions of the CLM model. This explains why the two-scale profile $\Omega$ must have the form of $\Omega_2$ in Theorem \ref{thm:two-scale}.

\subsection{Non-smooth initial data}
Other than the one-scale blowup \eqref{eqt:one-scale_blowup} with a smooth profile function $\Omega$, Elgindi--Jeong \cite{elgindi2020effects} also constructed one-scale self-similar blowup solutions with H\"older continuous profiles for the CLM model. More specifically, they found a family of self-similar solutions of the form 
\begin{equation}\label{eqt:one-scale_blowup_Holder}
\om(x,t) = (T-t)^{-1}\cdot \Omega_\alpha\left(\frac{x}{(T-t)^{1/\alpha}}\right),
\end{equation}
with $\alpha\in(0,1)$, where the profile (up to rescaling) is a one-point $C^\alpha$ function given by
\begin{equation}\label{eqt:one-scale_Holder_profile}
\Omega_\alpha(x) = \frac{-\sin(\alpha\pi/2)\sgn(x)|x|^\alpha}{1+2\cos(\alpha\pi/2)|x|^\alpha + |x|^{2\alpha}},
\end{equation}
which is odd in $x$ and smooth at any $x\neq 0$. Stability of this family of H\"older continuous profiles was also later established in \cite{elgindi2021stable,chen2021finite}, and hence an asymptotically self-similar blowup of the form \eqref{eqt:one-scale_blowup_asymptotic} can happen for odd initial data $\om_0$ such that $\om_0(x)\sim -\sgn(x)|x|^\alpha$ for $x\sim0$.

Therefore, it is also interesting to consider two-scale self-similar blowup of the CLM model with initial data that are degenerate but less regular at the origin. In particular, we may consider 
\begin{equation}\label{eqt:initial_Holder}
\om_0(x) = -\sgn(x)|x|^\beta p(x)
\end{equation}
for some $\beta> 1$ that is not an odd integer and for some smooth positive even function $p$ that decays sufficiently fast in the far field. One can again define $q(x)$ as in \eqref{eqt:p_q_definition}. When $\beta>4$, it is straightforward to check by the definition of the Hilbert transform that $q$ is at least $C^2$ near the origin and the expansion $q(x) = q(0) + O(x^2)$ still holds. Then, we can simply re-implement the proof of Theorem \ref{thm:two-scale_general} (but omit the tedious details here) to show that as $t\rightarrow T-0$, for $x=r(t)(T-t)^{1/2} + O((T-t)^{(\beta-1)/2})$,
\begin{equation}\label{eqt:two-scale_general_formula_Holder}
\om(x,t) = \frac{1}{(T-t)^{\beta/2}}\left(\Omega_2\left(\frac{x - r(t)(T-t)^{1/2}}{(T-t)^{(\beta-1)/2}}\right)+ o(1)\right),
\end{equation}
where $T$ and $r(t)$ are given as in Theorem \ref{thm:two-scale_general}, and $\Omega_2$ is again a traveling wave solution to \eqref{eqt:CLM} that can be rescaled to $-1/(1+x^2)$. When $\beta\in(2, 3)\cup(3,4]$, the corresponding function $q$ is less regular: for $x\sim0$, $q(x)=q(0) + O(|x|^{\beta-2})$ with $\beta\in(2, 3)\cup(3,4)$ and $q(x) = q(0) + O(|x|^2\ln|x|)$ with $\beta=4$. Nevertheless, one can slightly modify our proof and still show that the asymptotic expression \eqref{eqt:two-scale_general_formula_Holder} is valid for $x=r(t)(T-t)^{1/2} + O((T-t)^{(\beta-1)/2})$. Therefore, a two-scale self-similar finite-time blowup also happens with initial data of the form \eqref{eqt:initial_Holder} for $\beta>2$. 

The case $\beta =2$ is a bit more complicated but also interesting by itself, as in this case $q(x)\sim \ln(1/|x|)$ for $x\sim 0$. With a similar strategy, one can substitute the assumption \eqref{eqt:initial_Holder} into the explicit solution formula \eqref{eqt:explicit_solution} to prove that, for $x= r(t)\big((T-t)^{1/2}/|\ln(T-t)|^{1/2}\big) + O\big((T-t)^{1/2}/|\ln(T-t)|^{3/2}\big)$,
\[
\om(x,t) = \frac{|\ln(T-t)|}{T-t}\left(\Omega_2\left(\frac{x - r(t)\big((T-t)^{1/2}/|\ln(T-t)|^{1/2}\big)}{(T-t)^{1/2}/|\ln(T-t)|^{3/2}}\right)+ o(1)\right),
\] 
for some $T$ and $r(t)$ that can be determined in the proof. That is, a two-scale self-similar finite-time blowup still happens, but the ratio between the larger scale and the smaller scale is only $|\ln(T-t)|$.

As for $\beta\in(1,2)$, one should expect a one-scale self-similar blowup to happen. In fact, the non-smooth function \eqref{eqt:one-scale_Holder_profile} constructed in \cite{elgindi2020effects} is also an exact one-scale self-similar profile for $\alpha = \beta \in(1,2)$.

\subsection{On the blowup rates} We end this section with a remark on the blowup power $c_\om$. In the one-scale case in Theorem \ref{thm:one-scale}, since the profile $\Omega_1$ achieves its maximal absolute value at $z = a^2/c$, the maximum of $|\om(x,t)|$ is achieved approximately at $x = (T-t)a^2/c$. At this location, we have
\[
\begin{split}
\om((T-t)a^2/c,t) &= \frac{1}{T-t}\left(\Omega_1(a^2/c)+ O(T-t)\right)\sim \frac{1}{(T-t)},\\
\mtx{H}(\om)((T-t)a^2/c,t) &= \frac{1}{T-t}\left(\mtx{H}(\Omega_1)(a^2/c)+ O(T-t)\right)\sim \frac{1}{(T-t)}.
\end{split}
\]
Thus, by the equation \eqref{eqt:CLM}, the evolution of the maximum of $|\om(x,t)|$ is approximately
\[\partial_t\|\om(\cdot,t)\|_{L^\infty} \approx \frac{1}{(T-t)^2} \approx \|\om(\cdot,t)\|_{L^\infty}^2.\]
The nonlinear stretching $\om\mtx{H}(\om)$ is exactly quadratic in $\om$ at the maximum location of $\om$. This is consistent with $c_\om=-1$.

In the two-scale case in Theorem \ref{thm:two-scale}, the profile $\Omega_2$ achieves its maximal absolute value at $z = 0$, and so the maximum of $|\om(x,t)|$ is achieved approximately at $x = r(T-t)^{1/2}$. However, $\mtx{H}(\Omega_2)(0) = 0$. Hence, at this location, we have
\[
\begin{split}
\om(r(T-t)^{1/2},t) &= \frac{1}{(T-t)^{3/2}}\left(\Omega_2(0)+ O\big((T-t)^{1/2}\big)\right)\sim \frac{1}{(T-t)^{3/2}},\\
\mtx{H}(\om)(r(T-t)^{1/2},t) &= \frac{1}{(T-t)^{3/2}}\left(\mtx{H}(\Omega_2)(0)+ O\big((T-t)^{1/2}\big)\right)\sim \frac{1}{T-t}.
\end{split}
\]
In this case, the evolution of the maximum of $|\om(x,t)|$ is approximately
\[\partial_t\|\om(\cdot,t)\|_{L^\infty} \approx \frac{1}{(T-t)^{5/2}} \approx \|\om(\cdot,t)\|_{L^\infty}^{5/3}.\]
This is consistent with $c_\om=-3/2$. The nonlinear stretching $\om\mtx{H}(\om)$ has a much weaker alignment at the maximum location of $\om$, which however induces a much faster growth.

The argument above, however, does not seem to apply to the result in Theorem \ref{thm:two-scale_general} for a general $n\geq 2$. This is because the additive term $O(T-t)^{1/2}$ in the asymptotic formula of $\mtx{H}(\om)(x,t)$ in \eqref{eqt:two-scale_general_formula} is an overestimate of the error when $x$ is very close to $r(t)(T-t)^{1/2}$. One needs to carry out a more delicate estimate to see how $\mtx{H}(\om)$ scales in $\om$ at the maximum location of $|\om(x,t)|$ when $n\geq 2$.

\begin{figure}[!ht]
\centering
    \begin{subfigure}[b]{0.49\textwidth}
        \includegraphics[width=1\textwidth]{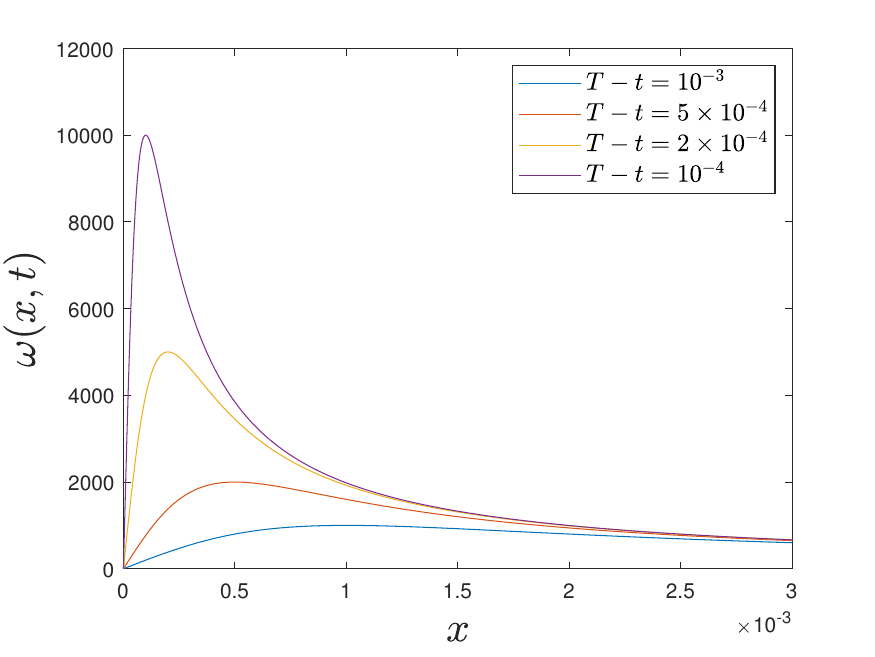}
        \caption{\small Evolution of a one-scale blowup}
    \end{subfigure}
    \begin{subfigure}[b]{0.49\textwidth}
        \includegraphics[width=1\textwidth]{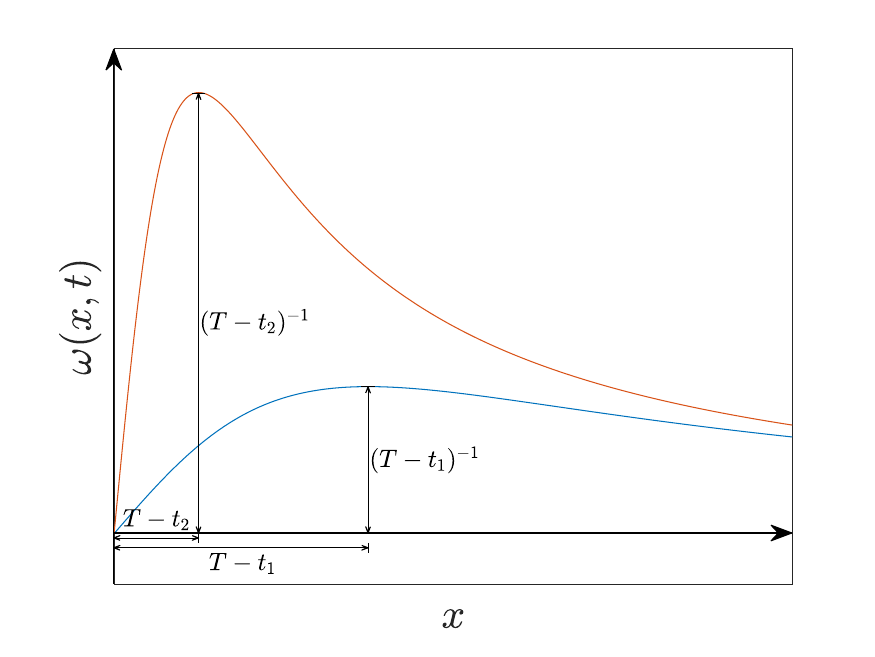}
        \caption{\small Demonstration of the scaling}
    \end{subfigure}
    \caption[One-scale]{Left: Evolution of a one-scale self-similar finite-time blowup with $\om_0(x)=-2x/(1+x^2)$. Right: Demonstration of the blowup and spatial scaling: $c_\om=-1$, $c_l=1$.}
    \label{fig:one-scale}
\end{figure}

\begin{figure}[!ht]
\centering
    \begin{subfigure}[b]{0.49\textwidth}
        \includegraphics[width=1\textwidth]{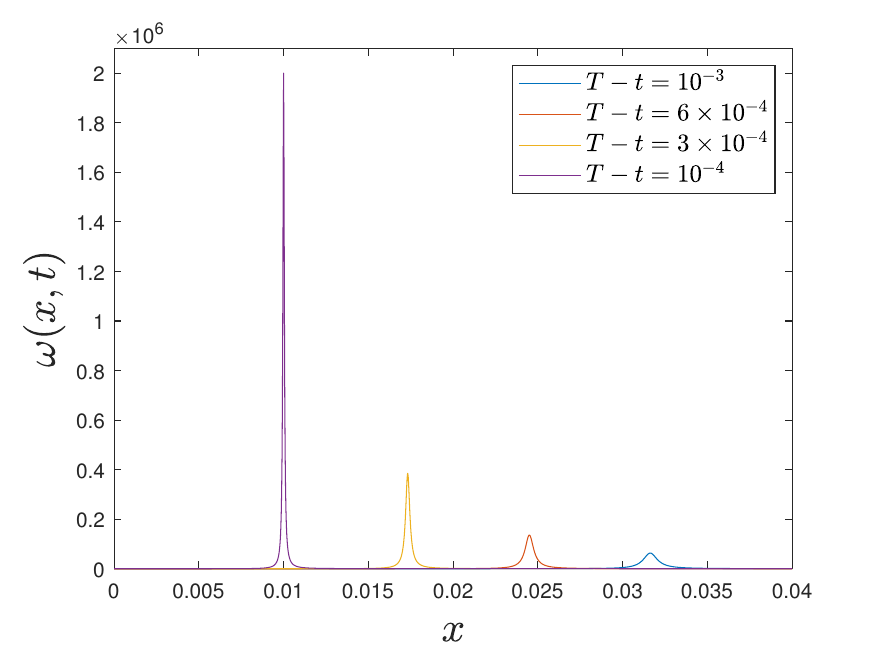}
        \caption{\small Evolution of a two-scale blowup}
    \end{subfigure}
    \begin{subfigure}[b]{0.49\textwidth}
        \includegraphics[width=1\textwidth]{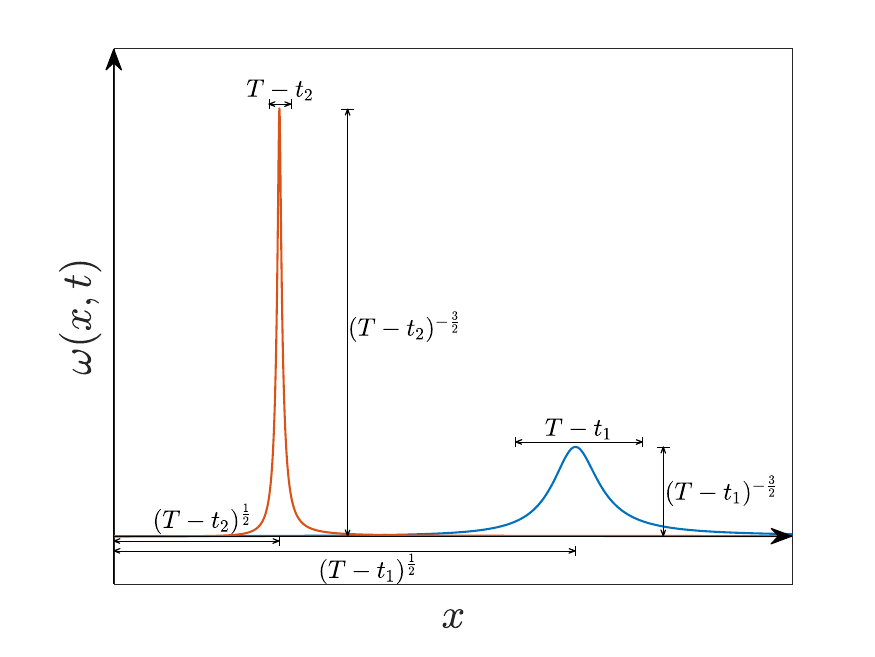}
        \caption{\small Demonstration of the scaling}
    \end{subfigure}
    \caption[Two-scale]{Left: Evolution of a two-scale self-similar finite-time blowup with $\om_0(x)=-2x^3/(4+x^4)$. Right: Demonstration of the blowup and spatial scaling: $c_\om=-3/2$, $c_l=1$, $c_s=1/2$ (corresponding to $n=1$).}
    \label{fig:two-scale}
\end{figure}

\section{A complex analysis}\label{sec:complex_analysis}
In this section, we carry out a complex analysis that transforms \eqref{eqt:CLM} from a non-local ODE to a local one. This kind of techniques has been commonly used in the study of the CLM model \cite{constantin1985simple,elgindi2020effects,lushnikov2021collapse}. We will use this approach to further investigate the nature of the multi-scale blowups.

\subsection{Complexification of the model} Applying the Hilbert transform to \eqref{eqt:CLM} and using Tricomi's identity $2\mtx{H}(\om\mtx{H}\om) = \mtx{H}(\om)^2 - \om^2$, we get a local system of the pair $(\om,\mtx{H}(\om))$:
\begin{equation*}
\begin{split}
\om_t &= \om\mtx{H}(\om),\\
\mtx{H}(\om)_t &= \frac{1}{2}\mtx{H}(\om)^2-\frac{1}{2}\om^2.
\end{split}
\end{equation*}
Consider the complex-valued function $\eta(x,t) := \om(x,t) + \iunit\mtx{H}(\om)(x,t)$. Then the equation of $\eta(x,t)$ is simply 
\[\eta_t = -\frac{\iunit}{2}\eta^2,\quad \eta(x,0) = \eta_0(x) := \om_0(x) + \iunit\mtx{H}(\om_0)(x).\]
From this equation one can easily obtain the explicit solution formula \eqref{eqt:explicit_solution} for $\om(x,t)$ and $\mtx{H}(\om)(x,t)$. This is how it was originally derived in \cite{constantin1985simple}. It may be even more convenient to consider the function $\zeta(x,t) = 1/\eta(x,t)$, which solves
\[\zeta_t = \frac{\iunit}{2},\quad \zeta(x,0) = \zeta_0(x) = 1/\eta_0(x).\]
It follows immediately that 
\[\zeta(x,t) = \zeta_0(x) + \frac{\iunit}{2}t.\]
If the pair $(\om,\mtx{H}(\om))$ blows up at some finite space-time point $(x_0,T)$, then $\eta$ also blows up at $(x_0,T)$, meaning that $x_0$ is a zero of $\zeta(x,T)$:
\[0 = \zeta(x_0,T) = \zeta_0(x_0) + \iunit T/2= \frac{\om_0(x_0)}{\om_0(x_0)^2 + \mtx{H}(\om_0)(x_0)^2} + \iunit\left(\frac{T}{2} - \frac{\mtx{H}(\om_0)(x_0)}{\om_0(x_0)^2 + \mtx{H}(\om_0)(x_0)^2}\right).\]
That is, $\om_0(x_0)=0$ and $\mtx{H}(\om_0)(x_0) = 2/T$. Conversely, if the point $x_0$ satisfies $\om_0(x_0)=0$ and $\mtx{H}(\om_0)(x_0) = 2/T$, then $\zeta(x_0,t)$ will become $0$ at $t=T$. This observation explains why a finite-time blowup will first happen at the point $x_0 \in \mathrm{S} =\{x: \om_0(x)=0, \mtx{H}(\om_0)(x)>0\}$ such that $\mtx{H}(\om_0)(x_0) = \sup\{\mtx{H}(\om_0)(x):x\in\mathrm{S}\}$. 

\subsection{Pole dynamics} Let us now extend the problem to the whole complex plane $\mathbb{C}$. Let $\eta$ be a meromorphic function on $\mathbb{C}$, whose poles only lie in the interior of the lower half-space $\mathbb{C}_-$. That is, $\eta(z)$ is holomorphic on the upper half-space $\mathbb{C}_+$. Suppose that $\eta$ decays to $0$ at the infinity. Then, it is well known that $\eta(x)$, the restriction of $\eta$ on the real line $\R$, satisfies
\[\mtx{H}(\real\eta(x)) = \imag\eta(x),\]
where $\real$ and $\imag$ denote the real part and the imaginary part, respectively.

Consider the ODE
\begin{equation}\label{eqt:eta_equation}
\eta_t = -\frac{\iunit}{2}\eta^2,\quad \eta(z,0) = \eta_0(z),
\end{equation}
where $\eta_0$ is holomorphic on the upper half-space $\mathbb{C}_+$. Then, up to any time $T$ such that the solution $\eta(z,t)$ is also holomorphic on $\mathbb{C}_+$ for $t\in[0,T)$, the real part $\om(x,t) = \real\eta(x,t)$ on $\R$ solves 
\[\om_t = \om\mtx{H}(\om),\quad \om(x,0) = \om_0(x) = \real\eta_0(x),\]
which is exactly the CLM model \eqref{eqt:CLM}. The poles of $\eta(z,t)$, which initially lie in the interior of $\mathbb{C}_-$, may travel upward and touch the real line $\R$ at some finite time $T$. Then a singularity of $\om$ will occur on $\R$ at time $T$. Hence, we can understand the blowup of $\om$ by studying the dynamics of the poles of $\eta$.

This is equivalent to studying the dynamics of the zeros of the function $\zeta(z,t) = 1/\eta(z,t)$ that is holomorphic on the lower half-space $C_-$. Again, $\zeta(z,t)$ solves 
\[\zeta_t = \frac{\iunit}{2},\quad \zeta(z,0) = \zeta_0(z) = 1/\eta_0(z),\]
and thus 
\[\zeta(z,t) = \zeta_0(z) + \frac{\iunit}{2}t.\]
Let $Z(t)$ be a moving zero of $\zeta(z,t)$ (i.e. a pole of $\eta(z,t)$) such that $\imag Z(0)<0$. That is, $\zeta(Z(t),t)\equiv0$. We then find that 
\[\zeta_z(Z(t),t)\cdot Z'(t) + \zeta_t(Z(t),t) = 0.\]
Note that we also have
\[\zeta_z(Z(t),t) = \zeta_{0,z}(Z(t)),\quad \zeta_t(Z(t),t)=\frac{\iunit}{2}.\]
Therefore, the motion of $Z(t)$ is governed by the equation
\begin{equation}\label{eqt:zero_dynamic}
Z'(t) = -\frac{\iunit}{2\zeta_{0,z}(Z(t))}.
\end{equation}
That is, the trajectory of $Z(t)$ is completely determined by $\zeta_0$ or by $\eta_0$. If the initial function $\zeta_0$ has a simple explicit form, we can also obtain $Z(t)$ by directly solving the equation 
\begin{equation}\label{eqt:zero_first_integral}
\zeta_0(Z(t)) + \frac{\iunit}{2}t = \zeta(Z(t),t) = 0.
\end{equation}

Despite the simplicity of the evolution of $\zeta(z,t)$, tracking the motion of its zeros can be highly non-trivial if the initial function $\zeta_0$ is complicated with many zeros. Nevertheless, we can still use a few basic examples to illustrate how the trajectories of the poles of $\eta$ with different initial data $\eta_0$ relate to different types of blowups of the CLM model. One can find in Figure \ref{fig:trajectory} plots of the pole trajectories in all these examples.

Before proceeding to the examples, we want to remark that the equation of $\eta(z,t)$ can be continued to all time regardless of the locations of its poles. However, after one or more poles of $\eta(z,t)$ cross the real line from below to above at some time $T$, $\eta(z,t)$ is no longer holomorphic on the upper half-space $\mathbb{C}_+$, and thus its restriction on the real line $\R$ no longer satisfies $\mtx{H}(\real\eta(x,t)) = \imag\eta(x,t)$. As a consequence, $\om(x,t) = \real \eta(x,t)$ no longer solves the CLM equation \eqref{eqt:CLM} after the blowup time $T$. Hence, we are only interested in the pole trajectories of $\eta(z,t)$ before they cross the real line.

\subsection{Example I: Exact one-scale self-similar blowup} Let us first choose 
\[\zeta_0(z) = -\frac{\iunit+z}{2},\quad \eta_0(z) = -\frac{2}{\iunit+z},\]
which corresponds to 
\[\om_0(x) = \real\left(-\frac{2}{\iunit+x}\right) = -\frac{2x}{1+x^2},\quad \mtx{H}(\om_0)(x)=\imag\left(-\frac{2}{\iunit+x}\right) = \frac{2}{1+x^2}.\]
Apparently, $\eta_0(z)$ is holomorphic on $\mathbb{C}_+$. The only pole of $\eta_0$ is $Z(0) = -\iunit$. By \eqref{eqt:zero_dynamic}, the dynamic of $Z(t)$ is given by
\[Z'(t) = -\frac{\iunit}{2\zeta_{0,z}(Z(t))} = \iunit\,,\]
that is, $Z(t) = \iunit(t-1)$. The first time $Z(t)$ touches the real line is when $t = 1$, at which time $\eta(z,t)$ has a pole on the real line and thus $\om(x,t)$ blows up. 

Note that $\eta(z,t)$ can be realized as
\[\eta(z,t) = \frac{2}{Z(t)-z} = \frac{2}{\iunit(t-1)-z} = \frac{1}{1-t}\cdot \frac{2}{-\iunit-z/(1-t)} = \frac{1}{1-t}\cdot \eta_0\left(\frac{z}{1-t}\right),\]
which is exactly self-similar in time, and so is its restriction on $\R$.

\subsection{Example II: Asymptotic one-scale self-similar blowup} This time, we choose 
\[\zeta_0(z) = \frac{2(1/2-\iunit-z)(1/2+\iunit+z)}{5(\iunit+z)} = \frac{1}{5}\left(\frac{1}{2(\iunit+z)} - 2(\iunit+z)\right),\]
which has two zeros, $Z_1(0) = 1/2-\iunit$ and $Z_2(0) = -1/2-\iunit$, in the lower half-space that are symmetric with respect to the imaginary axis. Correspondingly, we have
\[\eta_0(z) = \frac{5}{4}\left(\frac{1}{1/2-\iunit-z} - \frac{1}{1/2+\iunit+z}\right),\]
and
\[\om_0(x) = -\frac{30x + 40x^3}{25+24x^2+16x^4},\quad \mtx{H}(\om_0)(x) = \frac{50 + 40x^2}{25+24x^2+16x^4}.\]
Note that $\om_0$ satisfies all the assumptions in Theorem \ref{thm:one-scale}.

We use \eqref{eqt:zero_first_integral} to find that the zeros of $\zeta(z,t)$ satisfy
\[\frac{1}{5}\left(\frac{1}{2(\iunit+Z(t))} - 2(\iunit+Z(t))\right)=\zeta_0(Z(t)) = - \frac{\iunit}{2}t,\]
which yields
\[
Z(t) = 
\begin{cases}
\displaystyle \pm\frac{1}{2}\sqrt{1-\frac{25}{16}t^2} + \iunit\left(\frac{5}{8}t-1\right), &t\in[0,\frac{4}{5}],\\\\
\displaystyle \iunit \left(\pm \frac{1}{2}\sqrt{\frac{25}{16}t^2-1} + \frac{5}{8}t - 1\right), &t\in(\frac{4}{5},1].
\end{cases}
\]
We can see that the traveling of $Z(t)$ is much more complicated than in the previous example as it now consists of two phases. In the first phase $t\in[0,4/5]$, the two zeros 
\[Z_1(t) = \frac{1}{2}\sqrt{1-\frac{25}{16}t^2} + \iunit\left(\frac{5}{8}t-1\right),\quad Z_2(t) = -\frac{1}{2}\sqrt{1-\frac{25}{16}t^2} + \iunit\left(\frac{5}{8}t-1\right),\]
travel in curved lines from their starting points to the joining point $-\iunit/2$ where they merge into one point at $t=4/5$. Then, in the second phase $t\in[4/5,1]$, the two zeros,
\[Z_1(t) = \iunit \left(\frac{1}{2}\sqrt{\frac{25}{16}t^2-1} + \frac{5}{8}t - 1\right),\quad Z_2(t) = \iunit \left(- \frac{1}{2}\sqrt{\frac{25}{16}t^2-1} + \frac{5}{8}t - 1\right),\]
separate again and both travel on the imaginary axis. However, $Z_1(t)$ travels upward and reaches the real line $\R$ at $t=1$, while $Z_2(t)$ travels downward and away from the real line. 

Hence, the blowup of $\om(x,t)$ happens at $t=1$ when $Z_1(t)$, one of the two poles of $\eta(z,t)$, arrives at the origin. We can compute that
\begin{align*}
\eta(z,t) &= \frac{1}{\zeta(z,t)} = \frac{1}{\zeta_0(z) + \iunit t/2}\\
&= -\frac{5(\iunit+z)}{2(z-Z_1(t))(z-Z_2(t))} = -\frac{5(\iunit+z)}{4(\iunit(1-5t/8)+z)}\left(\frac{1}{z-Z_1(t)}+\frac{1}{z-Z_2(t)}\right).
\end{align*}
We can see that, as $t\rightarrow T-0$, the blowup of $\eta(z,t)$ near $z=0$ is dominated by the component
\[\frac{1}{z-Z_1}.\]
Since in the second phase $|Z_1(t)| = |\imag Z_1(t)| = O(1-t)$, the blowup is of only one spatial scale $1-t$. The contribution from the other pole $Z_2$ to the blowup at $z=0$ is only $O(1)$, which nevertheless makes the blowup not exactly self-similar on the real line.

\subsection{Example III: Two-scale blowup for $n=1$} Let us modify the previous example by choosing 
\[\zeta_0(z) = \frac{(1-\iunit-z)(1+\iunit+z)}{4(\iunit+z)} = \frac{1}{4}\left(\frac{1}{\iunit+z} - (\iunit+z)\right),\]
which again has two zeros, $Z_1(0) = 1-\iunit$ and $Z_2(0) = -1-\iunit$, in the lower half-space that are symmetric with respect to the imaginary axis. Correspondingly, we have
\[\eta_0(z) = \frac{2}{1-\iunit-z} - \frac{2}{1+\iunit+z},\]
and
\[\om_0(x) = -\frac{4x^3}{4+x^4},\quad \mtx{H}(\om_0)(x) = \frac{8+4x^2}{4+x^4}.\]
Note that $\om_0$ satisfies all the assumptions in Theorem \ref{thm:two-scale}. In particular, $\om(x)\sim -x^3$ near $x=0$.

We use \eqref{eqt:zero_first_integral} to find that the zeros of $\zeta(z,t)$ satisfy
\[\frac{1}{4}\left(\frac{1}{\iunit+Z(t)} - (\iunit+Z(t))\right)=\zeta_0(Z(t)) = - \frac{\iunit}{2}t,\]
that is 
\[Z(t) = \pm\sqrt{1-t^2} + \iunit(t-1),\quad t\in[0,1].\]
More precisely, we have
\[Z_1(t) = \sqrt{1-t^2} + \iunit(t-1),\quad Z_2(t) = -\sqrt{1-t^2} + \iunit(t-1).\]
Different from the previous example, in this case the traveling of the poles only has one phase before they reach the real line $\R$ for the first time. In particular, $Z_1(t)$ and $Z_2(t)$ merge into one point at the origin at $t=1$. Again, this implies the blowup of $\om(x,t)$ on $\R$ happens at $t=1$.

Let us only look at $Z_1(t)$ due to symmetry. If we represent the trajectory of $Z_1(t)$ in the $(x,y)$ coordinate, we find that 
\[Z_1(t) = (X_1(t),Y_1(t)) = \left(\sqrt{1-t^2}\,,\,t-1\right),\]
that is, 
\[Y_1(X_1) = \sqrt{1-X_1^2}-1.\]
The trajectory of $Z_1(t)$ is not a straight line, which explains why in this case the solution $\om(x,t)$ is not exactly self-similar. Moreover, we can see that $|Y_1(t)|\sim (1-t)$ converges to $0$ much faster than $|X_1(t)|\sim(1-t)^{1/2}$. This agrees with the two-scale blowup phenomenon described in Theorem \ref{thm:two-scale}.

We can also compute that 
\begin{align*}
\eta(z,t) &= \frac{1}{\zeta(z,t)} = \frac{1}{\zeta_0(z) + \iunit t/2}\\
&= -\frac{4(\iunit+z)}{(z-Z_1(t))(z-Z_2(t))} = -\frac{2(\iunit+z)}{\iunit(1-t)+z}\left(\frac{1}{z-Z_1(t)}+\frac{1}{z-Z_2(t)}\right).
\end{align*}
Consider the Laurent series of $\eta$ around $z = Z_1(t) = X_1(t)+\iunit Y_1(t)$ in terms of $\tilde{z} = (z-Z_1(t))/|Y_1(t)|$ for $|\tilde{z}|=O(1)$ and $t\in[0,1)$:
\begin{align*}
\eta(z,t) &= -\frac{2(\iunit+Z_1(t)+|Y_1(t)|\tilde{z})}{X_1(t) + |Y_1(t)|\tilde{z}}\left(\frac{1}{|Y_1(t)|\tilde{z}} + \frac{1}{2X_1(t) + |Y_1(t)|\tilde{z}}\right)\\
&=-\frac{2(\iunit + X_1(t)+\iunit Y_1(t))}{X_1(t)|Y_1(t)|}\cdot \frac{1}{\tilde {z}} + \frac{\iunit-X_1(t)+\iunit Y_1(t)}{X_1(t)^2} + O\left(\frac{1}{X_1(t)^2}\cdot \frac{|Y_1(t)\tilde{z}|}{X_1(t)}\right)\\
&= -\frac{2\iunit}{X_1(t)|Y_1(t)|}\cdot \frac{1}{\tilde {z}} + O\left(\frac{1}{X_1(t)^2} + \frac{1}{|Y_1(t)|}\right)\\
&= -\frac{2\iunit}{(1-t)^{3/2}(1+t)^{1/2}}\cdot \frac{1}{\tilde {z}} + O\left(\frac{1}{1-t}\right).
\end{align*}
Then, the real part of the restriction of $\eta$ on the real line is given by
\begin{align*}
\om(x,t) &= \mathrm{Re}(\eta(x,t)) \\
&= -\frac{2}{(1-t)^{3/2}(1+t)^{1/2}}\cdot \frac{1}{1 + \Big(\frac{x-(1-t^2)^{1/2}}{1-t}\Big)^2} + O\left(\frac{1}{1-t}\right)\\
&= -\frac{1}{(1-t)^{3/2}}\cdot \frac{\sqrt{2}}{1 + \Big(\frac{x-\sqrt{2}(1-t)^{1/2}}{1-t}\Big)^2} + O\left(\frac{1}{1-t}\right),
\end{align*}
which agrees with the asymptotic formula of $\om(x,t)$ in Theorem \ref{thm:two-scale} with $a=1$, $b=1/2$, $c=1/4$.

\subsection{Example IV: One-scale blowup away from the origin} Let us see what happens if the zeros of $\zeta_0$ are even farther away from each other. Consider 
\[\zeta_0(z) = \frac{(2-\iunit-z)(2+\iunit+z)}{4(\iunit+z)} = \frac{1}{2}\left(\frac{2}{\iunit+z} - \frac{\iunit+z}{2}\right),\]
whose two zeros are $Z_1(0) = 2-\iunit$ and $Z_2(0) = -2-\iunit$. Similarly, we have
\[\eta_0(z) = \frac{2}{2-\iunit-z} - \frac{2}{2+\iunit+z},\]
and
\[\om_0(x) = \frac{12x-4x^3}{25+x^4-6x^2},\quad \mtx{H}(\om_0)(x) = \frac{20+4x^2}{25+x^4-6x^2}.\]
We still have $\om_0(0)=0$, $\om_0'(0)\neq 0$, and $\mtx{H}(\om_0)(0)>0$. However, $x=0$ is not the only zero of $\om_0$. The other two zeros of $\om_0$ are $x=\pm\sqrt{3}$, and $\mtx{H}(\om_0)(\pm \sqrt{3}) = 2 > 4/5 = \mtx{H}(\om_0)(0)$. Therefore, the blowup $\om(x,t)$ should first happen at $x=\pm\sqrt{3}$ and $t=2/\mtx{H}(\om_0)(\pm \sqrt{3})=1$.

Let us verify this by looking at the traveling of $Z(t)$. We use \eqref{eqt:zero_first_integral} to find
\[\frac{1}{2}\left(\frac{2}{\iunit+Z(t)} - \frac{\iunit+Z(t)}{2}\right)=\zeta_0(Z(t)) = - \frac{\iunit}{2}t,\]
that is 
\[Z(t) = \pm\sqrt{4-t^2} + \iunit(t-1),\quad t\in[0,1].\]
More precisely, we have
\[Z_1(t) = \sqrt{4-t^2} + \iunit(t-1),\quad Z_2(t) = -\sqrt{4-t^2} + \iunit(t-1).\]
The first time $Z_1$ and $Z_2$ touch the real line is at $t=1$, which is before they can meet each other on the imaginary axis. Hence, the blowup of $\om(x,t)$ happens simultaneously at the two different points $x=\pm\sqrt{3}$ as expected. By representing the trajectory of $Z_1(t)$ in the $(x,y)$ coordinate,
\[Z_1(t) = (X_1(t),Y_1(t)) = \left(\sqrt{4-t^2}\,,\,t-1\right),\]
we see that $|X_1(t)-\sqrt{3}|\sim |Y_1(t)|\sim 1-t$ as $t\rightarrow 1-0$, which means that the blowup is of only one spatial scale around each of the blowup points.

\subsection{Example V: Two-scale blowup for $n=2$} Next, we provide an example of two-scale asymptotically self-similar blowup with more degenerate initial data at the blowup point. Let 
\[\zeta_0(z) =\frac{8(\iunit + z)^3}{3-9\iunit z -8z^2},\quad \eta_0(z)= \frac{3-9\iunit z -8z^2}{8(\iunit + z)^3}.\]
$\eta_0(z)$ has a triple pole $Z(0) = -\iunit$ in the lower half-space. Correspondingly, we have
\[\om_0(x) = -\frac{x^5}{(1+x^2)^3},\quad \mtx{H}(\om_0)(x) = \frac{3+10x^2+15x^4}{8(1+x^2)^3}.\]
It is easy to check that this $\om_0$ satisfies all the assumptions in Theorem \ref{thm:two-scale_general} for $n=2$, and thus a two-scale blowup with $c_\om=-5/2$, $c_l=2$, $c_s=1/2$ will happen. In fact, one can verify directly by the explicit solution formula \eqref{eqt:explicit_solution} that $\om(x,t)$ will develop a two-scale finite-time blowup at the origin such that, as $t\rightarrow T-0$,
\[\om(x,t) = \frac{1}{(T-t)^{5/2}}\left(\Omega\left(\frac{x - r(t)(T-t)^{1/2}}{(T-t)^2}\right)+ O((T-t)^{1/2})\right),\]
where
\[T=\frac{16}{3},\quad \Omega(z) = -\frac{3888}{6561+4096z^2},\quad r(t) = \frac{3}{4} - \frac{9}{16}(T-t).\]

As an exercise, the reader can also study the pole dynamics of the $\eta(z,t)$ in this example, which is quite complicated. Instead of presenting the detailed and tedious formula of the pole trajectories, we simply plot them in Figure \ref{fig:trajectory}(e): The initial triple pole, located at $-\iunit$, immediately branches into three simple poles, two of which travel in curved lines and merge into one at the origin at time $T=16/3$, inducing a two-scale blowup of $\om(x,t)$ on the real line. One can check that, as $t\rightarrow T-0$, either pole $Z(t) = X(t) +\iunit Y(t)$ approaching the origin satisfies $|X(t)|\sim(T-t)^{1/2}$ and $|Y(t)|\sim (T-t)^2$, that is, $|Y(t)| \sim |X(t)|^4$. This means the trajectory of $Z(t)$ for $n=2$ has a flatter profile near the origin than that of the two-scale blowup for $n=1$ in Example III.

\subsection{Example VI: Non-blowup traveling wave} In this last example, we show that a traveling wave solution $\om(x,t) = \om_0(x+rt)$ corresponds to a traveling pole of $\eta(z,t)$ whose trajectory is parallel to the real line. Consider the initial data
\[\zeta_0(z) = \frac{\iunit+z}{2\iunit},\quad \eta_0(z) = \frac{2\iunit}{\iunit+z},\]
which corresponds to 
\[\om_0(x) = \frac{2}{1+x^2},\quad \mtx{H}(\om_0)(x) = \frac{2x}{1+x^2}.\]
As we have argued in the previous section (see \eqref{eqt:traveling_wave} with $c=1$, $r=-1$), the solution to \eqref{eqt:CLM} with initial data $\om_0$ is a traveling wave
\[\om(x,t) = \om_0(x-t) = \frac{2}{1+(x-t)^2}.\]

According to \eqref{eqt:zero_dynamic}, the motion of the only zero $Z(t)$ of $\zeta(z,t)$ is governed by 
\[Z'(t) = -\frac{\iunit}{2\zeta_{0,z}(Z(t))} = 1,\quad Z(0) = -\iunit,\]
that is,
\[Z(t) = Z(0) + t = t-\iunit.\]
Hence, $Z(t)$ is traveling to the right on the line $\imag z=-\iunit$ that is parallel to the real line. This is consistent with the traveling wave of $\om(x,t)$.\\

\begin{figure}[!ht]
\centering
    \begin{subfigure}[b]{0.32\textwidth}
        \includegraphics[width=1\textwidth]{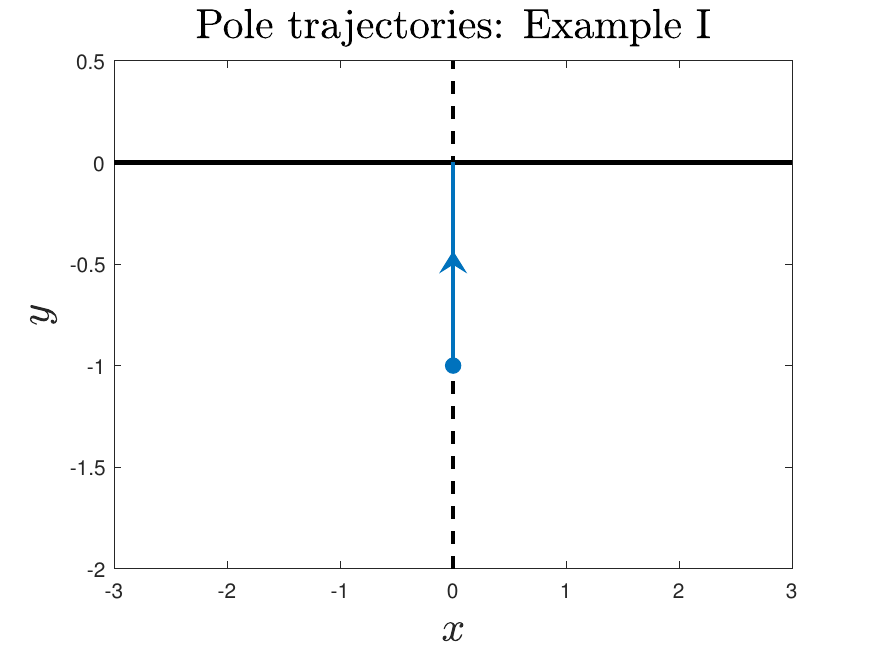}
        \caption{\small Example I: one-scale blowup}
    \end{subfigure}
    \begin{subfigure}[b]{0.32\textwidth}
        \includegraphics[width=1\textwidth]{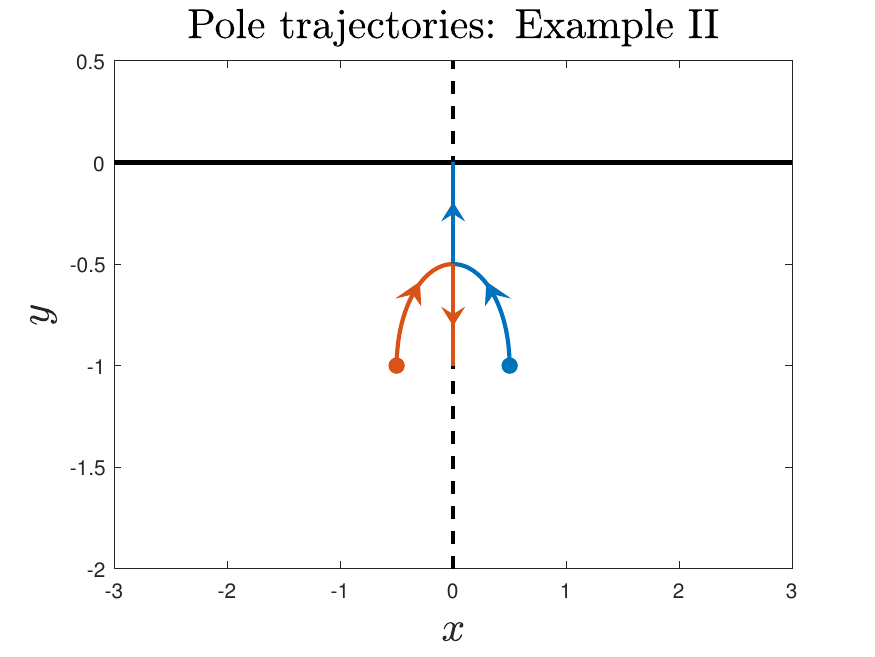}
        \caption{\small Example II: one-scale blowup}
    \end{subfigure}
    \begin{subfigure}[b]{0.32\textwidth}
        \includegraphics[width=1\textwidth]{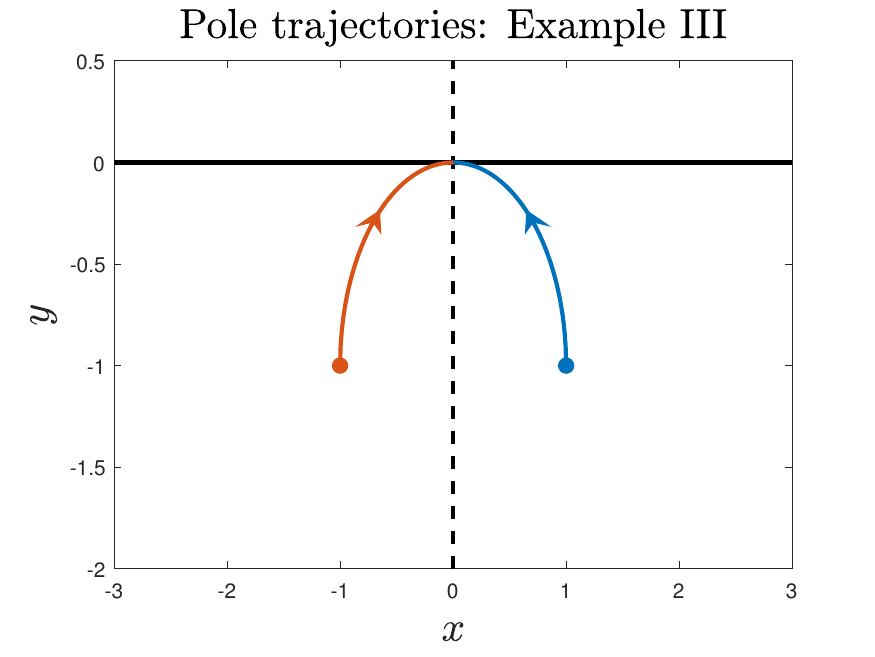}
        \caption{\small Example III: two-scale blowup}
    \end{subfigure}
    \begin{subfigure}[b]{0.32\textwidth}
        \includegraphics[width=1\textwidth]{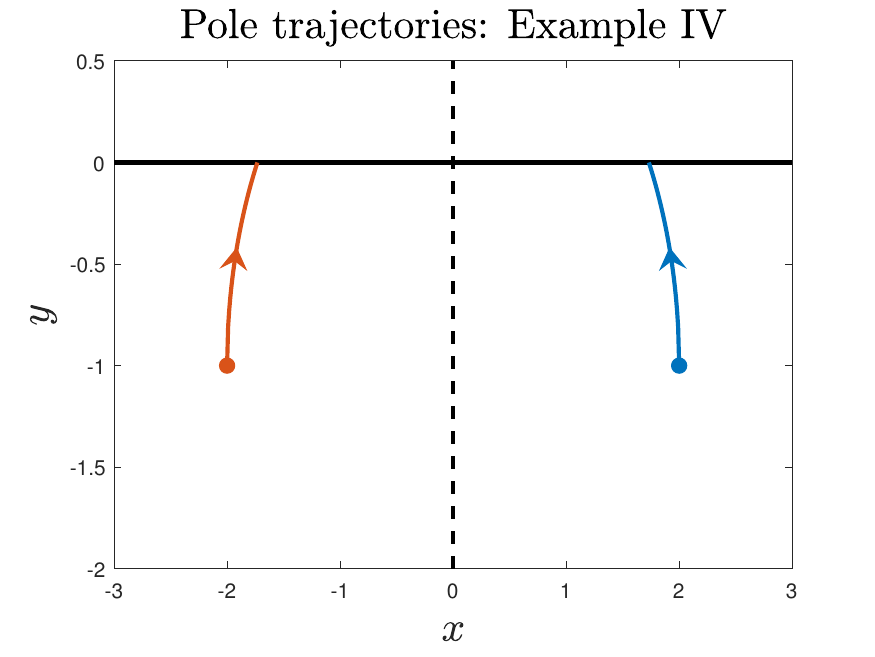}
        \caption{\small Example IV: one-scale blowup}
    \end{subfigure}
    \begin{subfigure}[b]{0.32\textwidth}
        \includegraphics[width=1\textwidth]{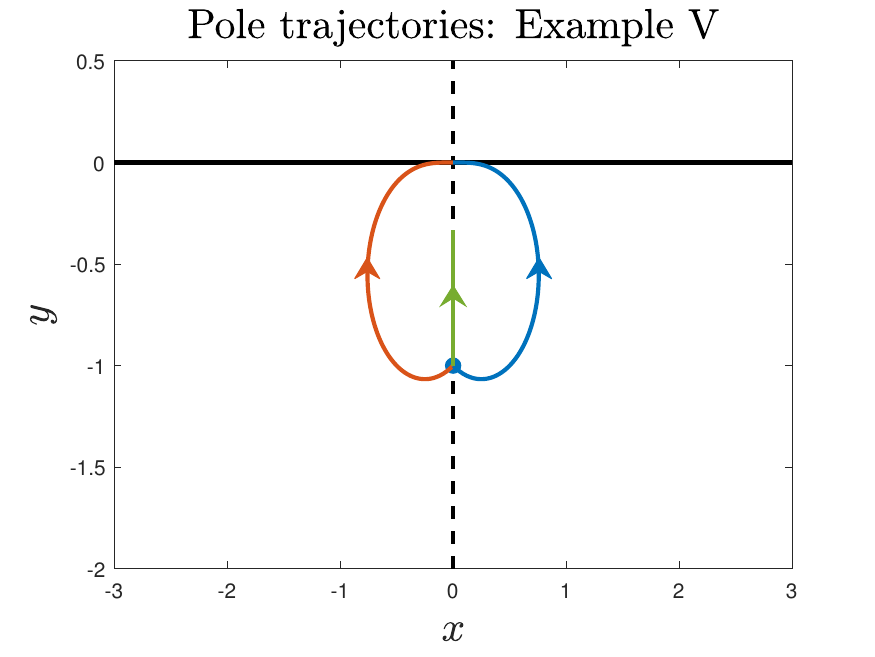}
        \caption{\small Example V: two-scale blowup}
    \end{subfigure}
    \begin{subfigure}[b]{0.32\textwidth}
        \includegraphics[width=1\textwidth]{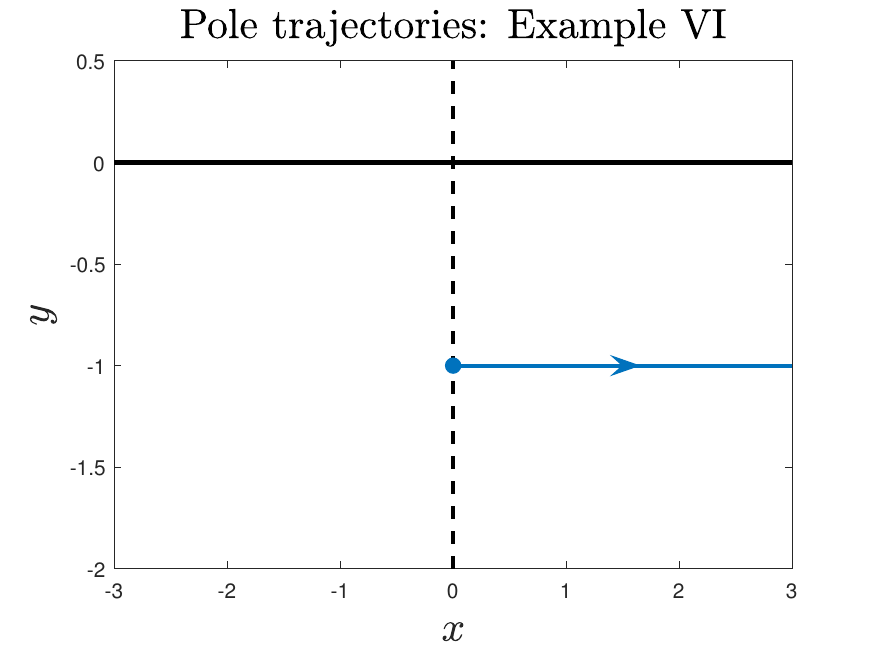}
        \caption{\small Example VI: traveling wave}
    \end{subfigure}
    \caption[Trajectories]{Plots of pole trajectories in all six examples.}
    \label{fig:trajectory}
\end{figure} 

\subsection{General analysis} As we have mentioned, analyzing the evolution of all poles of $\eta(z,t)$ for general initial data $\eta_0$ can be very difficult. Nevertheless, we can still perform a relatively general analysis on the local behavior of poles of $\eta(z,t)$ as they approach the real line $\R$. Let us write 
\[\eta_0(z) = u(x,y) + \iunit v(x,y),\]
that is, $u=\real \eta$ and $v=\imag \eta$. On the real line, we have $u(x,0) = \om_0(x)$ and $v(x,0) = \mtx{H}(\om_0)(x)$. Let $\om_0$ be odd in $x$ and sufficiently smooth. Suppose that $x=0$ is the unique zero of $\om_0$ in a small neighborhood $(-\delta,\delta)$ for some sufficiently small $\delta>0$, and for $|x|< \delta$,
\[\om_0(x) =-cx^{2n+1} + O(x^{2n+3}),\quad \mtx{H}(\om_0)(x) = a+bx^2 + O(x^4),\]
for some $a,c>0$, $b\in\R$ and some integer $n\geq 0$. We also assume that $b>0$ when $n\geq 1$. Note that $n=0$ corresponds to the one-scale case in Theorem \ref{thm:one-scale}, $n=1$ corresponds to the basic two-scale case in Theorem \ref{thm:two-scale}, and $n\geq 1$ corresponds to the general two-scale case in Theorem \ref{thm:two-scale_general}.

Furthermore, assume that the initial poles of $\eta_0$ all lie beyond the open disk $\mathbb{D}_\delta=\{z:|z|<\delta\}$. Then, $u(x,y)$ and $v(x,y)$ are conjugate harmonic functions in the disk $\mathbb{D}$ satisfying 
\[u_x = v_y,\quad u_y=-v_x.\]
Then one can easily show that, in the disk $\mathbb{D}$,
\begin{align*}
u(x,y) &= -cx^{2n+1} - 2bxy + h.o.t.,\\
v(x,y) &= a + bx^2 - (2n+1)cx^{2n}y + h.o.t.
\end{align*}
Here, $h.o.t.$ denotes higher order terms.

Consider the function $\zeta_0 = 1/\eta_0$. If we write 
\[\zeta_0(z) = A(x,y) + \iunit B(x,y),\]
then
\[A = \frac{u}{u^2 + v^2},\quad B = -\frac{v}{u^2 + v^2}.\]
Let us also assume that $\zeta_0$ has no pole in the disk $\mathbb{D}_\delta$. Then, $A$ and $B$ are also conjugate harmonic functions, and 
\begin{align*}
A(x,y) &= -\frac{c}{a^2}x^{2n+1} - \frac{2b}{a^2}xy + h.o.t.,\\
B(x,y) &= -\frac{1}{a} + \frac{b}{a^2}x^2 - \frac{(2n+1)c}{a^2}x^{2n}y + h.o.t. 
\end{align*}
In particular, we have
\begin{equation}\label{eqt:Ax_Bx_expansion}
\begin{split}
A_x(x,y) &= -\frac{(2n+1)c}{a^2}x^{2n} - \frac{2b}{a^2}y + h.o.t.,\\
B_x(x,y) &= \frac{2b}{a^2}x - \frac{2n(2n+1)c}{a^2}x^{2n-1}y + h.o.t. 
\end{split}
\end{equation}

Next, suppose that $Z(t)=X(t) +\iunit Y(t)$, a pole of $\eta(z,t)$, enters the disk $\mathbb{D}_\delta$ and will eventually arrive at the origin $(x,y)=(0,0)$ at the blowup time $T$. Recall the dynamic equation \eqref{eqt:zero_dynamic} of $Z(t)$, which in terms of $A,B$ writes
\[\big(A_x(X(t),Y(t)) + \iunit B_x(X(t),Y(t))\big)\cdot (X'(t) + \iunit Y'(t)) = -\frac{\iunit}{2},\]
that is,
\begin{align*}
A_x X'(t) - B_xY'(t) &= 0,\\
B_x X'(t) + A_xY'(t) &=-\frac{1}{2}.
\end{align*}
We have used the fact that $A_x = B_y$ and $A_y=B_x$ in the disk $\mathbb{D}_\delta$. It follows that
\begin{equation}\label{eqt:X_Y_dynamic}
X'(t) = -\frac{B_x}{2(A_x^2+B_x^2)},\quad Y'(t) = -\frac{A_x}{2(A_x^2+B_x^2)}.
\end{equation}
Hence, we can use the local expansions of $A_x,B_x$ to determine the asymptotic behavior of $Z(t)$ as it approaches the origin. We do it in two cases, $n=0$ and $n\geq 1$.

\subsubsection{Case $n=0$} Substituting \eqref{eqt:Ax_Bx_expansion} into \eqref{eqt:X_Y_dynamic} yields
\begin{align*}
X'(t) &= -\frac{(b/a^2)X(t) + h.o.t.}{c^2/a^4 + h.o.t.} = -\frac{a^2b}{c^2}X(t) + h.o.t.,\\
Y'(t) &= \frac{c/2a^2 + h.o.t.}{c^2/a^4 + h.o.t.} = \frac{a^2}{2c} + h.o.t.
\end{align*}
As we have assumed that $(X(t),Y(t))$ will arrive at the origin $(0,0)$ at $t=T$, the above implies that, when $t$ is sufficiently close to $T$, 
\[X(t) =0,\quad |Y(t)| \sim T-t.\]
We have used that $Y(t)<0$ for $t<T$. This type of trajectories of $Z(t)$ has appeared in Examples I, II and is consistent with the one-scale blowup described in Theorem \ref{thm:one-scale}.

\subsubsection{Case $n\geq 1$} Recall in this case we assume $a,b,c>0$. Substituting \eqref{eqt:Ax_Bx_expansion} into \eqref{eqt:X_Y_dynamic} yields
\begin{align*}
X'(t) &= -\frac{(b/a^2)X(t) + h.o.t.}{(4b^2/a^4)(X(t)^2 + Y(t)^2) + h.o.t.} = -\frac{a^2}{4b}\cdot \frac{X(t)}{X(t)^2+Y(t)^2} + h.o.t.,\\
Y'(t) &= \frac{((2n+1)c/2a^2)X(t)^{2n} + (b/a^2)Y(t) + h.o.t.}{(4b^2/a^4)(X(t)^2 + Y(t)^2) + h.o.t.} = \frac{((2n+1)a^2c/2)X(t)^{2n} + a^2bY(t)}{4b^2(X(t)^2 + Y(t)^2) } + h.o.t.
\end{align*}
We first note that 
\[\frac{\diff Y}{\diff X} = \frac{Y'(t)}{X'(t)} = -\frac{(2n+1)c}{2b}X(t)^{2n-1} - \frac{Y(t)}{X(t)} + h.o.t.\]
This equation, together with the assumption that $X(T)=Y(T)=0$, imply that
\[Y(t) = -\frac{c}{2b}X(t)^{2n} + h.o.t.\]
Substituting this into the ODE of $X(t)$ yields 
\[X'(t) = -\frac{a^2}{4b} X(t)^{-1} + h.o.t.,\]
which then implies 
\[X(t)^2 = \frac{a^2}{2b}(T-t) + h.o.t.\]

In summary, as $t\rightarrow T-0$, we have
\[|X(t)| \sim (T-t)^{1/2},\quad |Y(t)| \sim (T-t)^n.\]
This class of trajectories has appeared in Examples III, V and is consistent with the two-scale blowup described in Theorem \ref{thm:two-scale_general} with $c_s=1/2$, $c_l=n$. That is, $|X(t)|$ corresponds to the larger scale that is always $(T-t)^{1/2}$, while $|Y(t)|$ corresponds to the smaller scale $(T-t)^{n}$ that can be arbitrarily small.

\subsection*{Acknowledgement} The authors are supported by the National Key R\&D Program of China under the grant 2021YFA1001500.

\bibliographystyle{myalpha}
\bibliography{reference}

\end{document}